\pdfoutput = 1

\documentclass[english]{article}
\usepackage{lmodern}
\usepackage[T1]{fontenc}
\usepackage[latin9]{inputenc}
\usepackage{geometry}
\usepackage{color}
\usepackage{babel}
\usepackage{float}
\usepackage{amsmath}
\usepackage{amsthm}
\usepackage{amssymb}
\usepackage{graphicx}
\usepackage[authoryear]{natbib}

\makeatletter

\newcommand{\noun}[1]{\textsc{#1}}
\providecommand{\tabularnewline}{\\}
\floatstyle{ruled}
\newfloat{algorithm}{tbp}{loa}
\providecommand{\algorithmname}{Algorithm}
\floatname{algorithm}{\protect\algorithmname}

\theoremstyle{plain}
\newtheorem{thm}{\protect\theoremname}[section]
\theoremstyle{definition}
\newtheorem{defn}[thm]{\protect\definitionname}
\theoremstyle{plain}
\newtheorem{assumption}[thm]{\protect\assumptionname}
\theoremstyle{definition}
\newtheorem{example}[thm]{\protect\examplename}
\theoremstyle{plain}
\newtheorem{lem}[thm]{\protect\lemmaname}
\theoremstyle{plain}
\newtheorem{prop}[thm]{\protect\propositionname}

\makeatother

\providecommand{\assumptionname}{Assumption}
\providecommand{\definitionname}{Definition}
\providecommand{\examplename}{Example}
\providecommand{\lemmaname}{Lemma}
\providecommand{\propositionname}{Proposition}
\providecommand{\theoremname}{Theorem}

\begin{document}

\title{Data-driven satisficing measure and ranking}
\author{Wenjie Huang \footnote{wenjie\_huang@u.nus.edu, Department of Industrial Systems Engineering and Management at National University of Sinagpore.} }
\maketitle
\begin{abstract}
We propose an computational framework for real-time risk assessment
and prioritizing for random outcomes without prior information on
probability distributions. The basic model is built based on satisficing
measure (SM) which yields a single index for risk comparison. Since
SM is a dual representation for a family of risk measures, we consider
problems constrained by general convex risk measures and specifically
by Conditional value-at-risk. Starting from offline optimization,
we apply sample average approximation technique and argue the convergence
rate and validation of optimal solutions. In online stochastic optimization
case, we develop primal-dual stochastic approximation algorithms respectively
for general risk constrained problems, and derive their regret bounds.
For both offline and online cases, we illustrate the relationship
between risk ranking accuracy with sample size (or iterations).\\
\emph{}\\
\emph{Keywords:} Risk measure; Satisficing measure; Online stochastic
optimization; Stochastic approximation; Sample average approximation;
Ranking; 
\end{abstract}

\section{Introduction}

Risk assessment is the process where we identify hazards, analyze
or evaluate the risk associated with that hazard, and determine appropriate
ways to eliminate or control the hazard. Risk assessment techniques
have been widely applied in many area including quantitative financial
engineering (\citealt{Krokhmal2002}), health and environment study
(\citealt{Zhang2012,VanAsselt2013}), transportation science (\citealt{Liu2017}),
etc. \citet{Paltrinieri2014} point out that traditional risk assessment
methods are often limited by static, one-time processes performed
during the design phase of industrial processes. As such they often
use older data or generic data on potential hazards and failure rates
of equipment and processes and cannot be easily updated in order to
take into account new information, giving a more complete view of
the related risks. This failure to account for new information can
lead to unrecognized hazards, or misunderstandings about the real
probability of their occurrence under current management and safety
precautions. With the rapid development of computational intelligence
and corresponding decision support system, as well as the launch of
``Big data'' era, nowadays, new risk assessment technique should
allow decision maker to update the assessment results by observing
new information or data and realize quick response to dynamic environment.
In this paper, we develop a satisficing measure based model to assess,
compare and ranking random outcomes, and propose both online and offline
data-driven computational frameworks. We validate the methods both
theoretically and experimentally. The title of this paper ``Data-driven''
means that the probability distribution of the randomness is not available
in our settings, and we conduct risk assessment and ranking, only
based on observed empirical data. 

The core of real-time assessment allows us to update the assessment
results by observing new information or data and realize quick response
to dynamic environment. Two most commonly used real-time assessment
techniques are Hidden Markov model (HMM) and Bayesian network introduced
by \citet{Ghahramani2001}. HMM models used in \citet{Tan2008,Li2009,Yu-Ting2014,Haslum2006}
can combine both external and internal threats in network security
systems, and the experiment results show that such method can improve
the accuracy and reliability of assessment than the statics approaches.
Bayesian approaches used in \citet{Sandoy2006} can solve problems
generated by introducing or not introducing an underlying probability
model. The pairwise comparison theory (also called Bradley-Terry-Luce
(BTL) model) widely used in study the preference in decision making
was introduced in \citet{Bradley1952,Luce2005}. A Bayesian approximation
method with BTL model in \citet{Weng2011}, is proposed for online
ranking in team performance. However, these methodologies do not consider
assessing the ``risk'' of the systems and have not provided a formal
definition of risk that can be uniformly applied to wide range of
systems. 

In financial mathematics, \textcolor{black}{``risk measure'' is
defined as a mapping function from a probability space to real number.
Some fundamental research has been conducted motivating the definition
of risk measure. In \citet{Ruszczynski2006}, the definitions and
conditions for convex and coherent risk measures are developed, and
conjugate duality reformulation of risk functions is proposed. \citet{Rockafellar2000}
give the detailed arguments on the most widely investigated coherent
and law invariant risk measure-Conditional value-at-risk (CVaR) with
corresponding reformulation and optimization problem illustration.
CVaR is widely involved in optimization under uncertainty, and helps
to improve the reliability of solutions against extremely high loss.
\citet{Krokhmal2002} study the portfolio optimization problem with
CVaR objective and constraints, and corresponding reformulation and
discretization are demonstrated. In \citet{Dai2016}, robust version
CVaR is applied in portfolio selection problem, where sampled scenario
returns are generated by a factor model with some asymmetric and uncertainty
set. In \citet{Noyan2013}, multivariate CVaR constraint problem is
studied based on polyhedral scalarization and second-order stochastic
dominance, and a cut generation algorithm is proposed, where each
cut is obtained by solving a mixed integer problem. \citet{Xu2014}
reformulate a stochastic nonlinear complementary problem with CVaR
constraints, and propose a penalized smoothing sample average approximation
algorithm to solve the CVaR-constrained stochastic programming. \citet{Shapiro2013}
gives Kusuoka representation of law invariant risk measures. The basic
idea is to use a class of CVaR constraints to formulate any coherent
and law invariant risk measure. }

In practice, the above classical risk measures have shortcomings in
real applications. First, classical risk measures like CVaR require
the decision maker to \textcolor{black}{specify his own risk tolerance
parameters in order to accurately capture the risk preference of decision
maker as well as provide an exact mathematical formulation of risk.
However, this process is very different to realize in practice, mentioned
in \citet{Delage2015,armbruster2015decision}. Secondly, we consider
that the ``risk'' of random variables are compared, based on metric
on how the random outcome exceeds certain risk measure. If the random
outcome exceeds the risk measure, we consider these outcomes are in
risky region. Obviously, classical risk measures are not appropriate
for conducting risk comparison for random outcomes under different
probability distributions. For example, suppose random variable $X$
dominates $Y$ in first-order. We can conjuncture that there exist
realizations of $X$ and $Y$, such that realizations of $X$ has
higher value than realization of $Y$, but realization of $Y$ exceeds
its value-at-risk while realization $X$ does not. Thus, comparing
the value of risk measure is failed to identify which random variable
is more ``risky''. Thus, new models are required specifically for
risk assessment and comparison.}

\textcolor{black}{Recently, new metrics on risk are developed including
satisficing measure and aspirational preference measure. In \citet{Brown2009},
satisficing measure evaluating the quality of financial positions
based on their ability to achieve desired financial goals can be show
to be dual to a class of risk measures. Such target are often much
more natural to specify than risk tolerance parameters, and ensure
robust guarantees. In\citet{Brown2012}, Aspirational preference measure
is developed as an expanded case for satisficing measure that can
handle ambiguity without a given probability distribution, moreover,
it possesses more general properties than satisficing measure. These
target-based risk measures yield a single index to metric the risk
of random outcomes. Besides, the index are normally contained in a
fixed and closed interval, which can be used to rank the risk of random
outcomes. In \citet{Chen2014}, satisficing measure is applied in
studying the impact of target on newsvendor decision. For practical
use, there still exists space for improvement for satisficing measure.
First, current measures are developed based on some simple and specific
risk measures like CVaR, rather than adapt themselves to general convex
or coherent risk measures. Second, the models rely on the full knowledge
of the probability distribution of the random variable, which are
hard to collect in practice. To make the full potential of these models
for practical use, in this paper,} we follow the ideas in \citet{Postek2015,Ben-Tal2007},
derive computationally tractable counterparts of distributionally
robust constraints on risk measures, by using the optimized certainty
equivalent framework. Thus we provide a general formulation of satisficing
measures that can cover range of risk measures. 

Lacking the information of probability distribution of uncertainty
is normally an important issue in risk management and robust optimization.
\citet{Delage2010} connect distributionally robust optimization with
data-driven techniques, and propose a model that describes uncertainty
in both the distribution form (discrete, Gaussian, exponential, etc.)
and moments (mean and co-variance matrix). \citet{Brown2006} provides
a comprehensive and integrated view between convex and coherent risk
measure with robust optimization, and provide probability guarantee
centered on data-driven approach. Instead of using data to estimate
the uncertainty set, we derive the robust counterpart of distributionally
robust formulation of risk, and then adopt online stochastic optimization
(see \citealt{Bubeck2011,Shalev-Shwartz2011}) approach for data-driven
optimization. Such computational framework is more efficient and easy
to analyze the theoretical properties. Research has been done in extending
online unconstrained stochastic optimization methods in constrained
optimization or risk-aware optimization. In \citet{Mahdavi2012},
online stochastic optimization with multiple objective is studied,
the idea is to cast the stochastic multiple objective optimization
problem into a constrained optimization problem by choosing one function
as the objective and try to bound other objectives by appropriate
thresholds. Projected gradient method and efficient primal-dual stochastic
algorithm are developed to tackle such problem. In \citet{Duchi2011},
a new family of subgradient methods has been presented that dynamically
incorporate knowledge of the geometry of the data observed in earlier
iterations to perform more informative gradient-based learning. In
\citet{Bardou2009}, stochastic approximation approach has been applied
to estimate CVaR in data-driven optimization problems. Moreover, in
\citet{Carbonetto2009}, studies has been conducted to develop stochastic
interior-point algorithm to solve constrained problem. In this paper,
our online algorithm extends the stochastic approximation of CVaR
in \citet{Bardou2009} to constrained optimization problem with general
convex risk measures.

In this paper, we develop a general framework for data-driven risk
assessment and ranking. The key contributions lay in three aspects.
\begin{enumerate}
\item We build a risk constrained satisficing measure model which can yield
a single index for risk comparison. Using optimized certainty equivalent
formulation, our model can cover a wide range of risk measures. 
\item Without knowledge on the probability distribution, our data-driven
computational methods are developed only based on observations. We
consider both in offline and online cases. In offline case, sample
average approximation (SAA) approach were applied to perform convergence
analysis for the feasible region and validation analysis for optimal
solution. In online case, we develop a primal-dual algorithm to solve
the problem. We figure out the regret bound and its relationship with
iteration number.
\item We check the validation of risk ranking results both in offline and
online cases. Explicitly, we want to check given all the information
of a batch of random variables and their true risk ranking. How close
is the ranking computed from data-driven method to the true ranking.
As the sample size grows, we show the convergence rate of ranking
results from SAA and online algorithm to the true underlying result
in probability. 
\end{enumerate}
The remainder of the paper is organized as follows. In Section 2 we
establish the necessary preliminaries and introduce some basic methodologies.
Section 3 illustrates the general model for offline problem, sample
average approximation analysis, and validation of offline ranking
results. Section 4 presents efficient primal-dual algorithm to solve
online risk assessment problem, with validation of online ranking
results. Finally, we conclude the paper with open questions in Section
5.

\section{Preliminaries}

This section introduces preliminary concepts and notation to be used
throughout the paper.

\subsection{Risk measure}

Define a certain probability space $(\Omega,\mathcal{F},\,P_{0})$,
where $\Omega$ is a sample space, $\mathcal{F}$ is a $\sigma-$algebra
on $\Omega$, and $P_{0}$ is a probability measure on $(\Omega,\mathcal{\,F})$
. We concern throughout with random variables in $\mathcal{L}=L_{\infty}(\Omega,\mathcal{\,F},\,P_{0})$,
the space of essentially bounded $\mathcal{F}-$measurable functions.
if let $\mathcal{X}$ denote the linear space of $\mathcal{F}-$measurable
functions $X:\,\Omega\rightarrow\mathbb{R}$. We first define concept
of \emph{risk measure }based on the definition in \citet{Ruszczynski2006}
as a function $\rho$, which assigns to an uncertain random variable
$X$ a real value $\mu(X)$. Formally, 

A risk measure is a mapping $\rho:\mathcal{X}\rightarrow\mathbb{R\cup}\{+\infty\}\cup\{-\infty\}$
if $\mu(0)$ is finite and if $\mu$ satisfies the following conditions
for all $X,\,Y\in\mathcal{X}.$ For $X,\,Y\in\mathcal{X},$ the notation
$Y\succeq X$ means that $Y\text{(\ensuremath{\omega})\ensuremath{\geq}\ensuremath{X(\ensuremath{\omega})}}$
for all $\omega\in\Omega$. The following four properties of risk
functions are important throughout our analysis:

(A1) Monotonicity: If $Y\succeq X$, then $\rho(X)\geq\rho(Y)$.

(A2) Transition Invariance: If $r\in\mathbb{R}$, then $\rho(X+r)=\rho(X)-r.$

(A3) Convexity: $\rho(\lambda X+(1-\lambda)Y)\leq\lambda\rho(X)+(1-\lambda)\rho(Y),$
for $0\leq\lambda\leq1.$

(A4) Positive Homogeneity: If $\lambda\geq0$, then $\rho(\lambda X)=\lambda\rho(X).$\\
These conditions were introduced, and real valued functions $\rho:$
$\mathcal{X}\rightarrow\mathbb{R}$ satisfying (A1)-(A4) were called
\emph{coherent measures of risk} in the pioneering paper (\citealt{Artzner1999}).
In fact, (A3) is equivalent to weaker requirement called Quasi convexity:
$\rho(\lambda X+(1-\lambda)Y)\leq\max(\rho(X),\,\rho(Y))$, for $0\leq\lambda\leq1$,
when the risk measure is positive homogeneous. This property reveals
the ``diversification'' preference of actual decision maker. Recall
that a risk measure is law invariant if it only depends on the distributions
of the random variables in question and not on the underlying probability
space, i.e.

\[
\rho(X)=\rho(Y),\,\forall X=_{D}Y,
\]
where $=_{D}$ denotes equality in distribution. We present general
convex risk measure into distributionally robust formulation\emph{
}based on the definition in \citet{Ruszczynski2006}. The distributionally
robust formulation of\emph{ }risk measure enables us to construct
any risk measure on probability space $\mathcal{L}$ by choosing different
convex and proper function embedded in the risk measure.\textcolor{black}{{}
Given a probability} space $\mathcal{L}_{p}(\Omega,\,\mathcal{F},\,P)$
is associated with its dual space $\mathcal{L}_{q}(\Omega,\,\mathcal{F},\,P)$,
satisfying $q\in(1,+\infty]$ and $1/p+1/q=1$. For $\omega\in\Omega$,
$X\in\mathcal{L}_{p}$ and $Q\in\mathcal{L}_{q}$, the expectation
is defined by their scalar product as
\[
\mathbb{E}^{Q}[X]=\int_{\Omega}Q(\omega)X(\omega)dP(\omega).
\]
Based on the results derived in \citet{Foellmer2011,Foellmer2013},
it is proved that any convex risk measure $\rho$ on random outcome
$X$ can be represented into a robust representation as
\begin{equation}
\rho(X)=\sup_{Q\in\mathcal{P}}\left\{ \mathbb{E}^{Q}[X]-h(Q)\right\} ,\label{Convex risk measure}
\end{equation}
where$\mathcal{P}$ is the set of all probability measures on $\Omega$,
and function $h:\mathcal{L}_{q}\rightarrow\mathbb{R}$ is a proper
and convex penalty function on $\mathcal{P}$ satisfying $\inf_{Q\in\mathcal{P}}h(Q)=0$,
if $\rho$ is proper, lower semicontinuous, and convex function. When
$h(Q)=0,\,\forall Q\in\mathcal{P}$, then any coherent risk measure
can be represented by (\ref{Convex risk measure}), i.e.
\[
\rho(X)=\sup_{Q\in\mathcal{P}}\,\mathbb{E}^{Q}[X].
\]

\subsection{Target-based measure}

Target-based decision making is first investigated in \citet{Simon1955,Simon1959}
where the concept of \emph{satisficing} and \emph{aspirational levels}
are introduced to evaluate the preference and action of decision makers.
Recently new risk measures developed in literature that are able to
evaluate the quality of positions with random outcomes based on their
ability to achieve desired target in \citet{Shalev2000,Sugden2003,DeGiorgi2011,Koszegi2006}.
\emph{Satisficing measure} proposed by \citet{Brown2009} is the most
recent invention of target-based measure. Moreover, this paper further
implies that optimization of these measures can be approached using
computationally tractable tools from convex optimization, in contrast
to the difficult, combinatorial problems that plague optimization
of value-at-risk and related measures. Finally, these satisficing
measures have a separation property which allows us to compute a single
``tangent'' portfolio regardless of the desired expected value.
Such value is quite useful for risk comparison and ranking. The definition
of satisficing measure we use is from \citet[Definition 1]{Brown2009}.
Let $(\Omega,\mathcal{\,F},\,P)$ be a probability space and let $\mathcal{\mathcal{L}}$
be a set of random variable on $\Omega$. The decision maker has an
aspiration level $\tau$ as an target. Given an uncertain payoff $X\in\mathcal{\mathcal{L}}$,
define target premium $V$ to be the excess payoff above $\tau$,
i.e., $V=X-\tau\in\mathcal{V}$. We assume that $\mathcal{V=\mathcal{L}}$.
In other words, we will assume each of the payoffs $X\in\mathcal{L}$
already has the aspiration level embedded within it, and suppress
the notation $\tau$ in the following definition. 
\begin{defn}
\label{Definition 2.2} A function $\mu:\,\mathcal{\mathcal{L}\rightarrow}[0,\,\bar{p}],$
where $\bar{p}\in\left\{ 1,\,\infty\right\} ,$ is a satisficing measure
defined on the target premium if it satisfies the following axioms
for all $X,\,Y\in\mathcal{L}$:
\end{defn}

\begin{itemize}
\item \emph{Attainment content:} If $X\geq0$, then $\mu(X)=\bar{p.}$
\item \emph{Non-attainment apathy:} If $X<0$, then $\mu(X)=0.$
\item \emph{Monotonicity: }If $X\geq Y,$ then $\mu(X)\geq\mu(Y).$
\item \emph{Gain continuity:} $\lim_{a\rightarrow0^{+}}\mu(X+a)=\mu(X).$
\end{itemize}
Traditional risk measure like Conditional value-at-risk derives a
mapping that yields a risk outcome of random variable based on certain
quantile level; however, the risk measure and comparison result will
varies by selecting different quantile level, while a more natural
approach is to provide a framework for measuring the quality of risky
positions with respect to their ability to satisfy a certain target
$\tau$, as a metric for measuring how ``risky'' a random outcomes
is. This concept has the advantage that aspiration levels are often
natural for decision makers to specify, as opposed to the risk-tolerance
type parameters, which can be difficult to intuitively understand
and hard to appropriately specify, that are necessary for many other
approaches (risk measures, utility functions, etc.). \textcolor{black}{Moreover,
the optimal satisficing value is normalized in a bounded interval,
usually $[0,\,1]$, which is natural and convenient to illustrate
the ranking result. Finally, any satisficing measure can be reformulated
mathematically as a dual problem of its corresponding risk measure,
so that in \citet{Brown2012}, aspirational preference measure is
developed as an expanded case for satisficing measure that can handle
ambiguity without a given probability distribution; moreover, it possesses
more general properties than satisficing measure. Computationally,
these two risk metrics lead to the topics on solving risk constrained
optimization problems, for example, CVaR constrained problem.}

\section{Model}

In this section, we develop the general satisficing measure and ranking
model. \citet{Rockafellar2013} introduce the concept called Risk
quadrangle that any risk measure can be portrayed on a higher level
as generated from penalty-type expressions of ``regret'' about the
mix of potential outcomes. Specifically, define $(\Omega,\mathcal{\,F},\,P)$
be a probability space, let $\mathcal{\mathcal{L}}$ be a set of random
variable on $\Omega$, and the random variable $X\in\mathcal{L}$.
Then any risk measure can be defined as following trade-off formula.
\[
\mu(X)=\min_{C\in\mathbb{R}}\left\{ C+\mathcal{\mathcal{U}}(X-C)\right\} ,
\]
Where $\mathcal{U}$ is mapping $\mathcal{L}\rightarrow\mathbb{R}$,
called\emph{ measure of regret}. Regret comes up in penalty approaches
to constraints in stochastic optimization and, in mirror image, corresponds
to measures of utility. We first recall a class of certainty equivalents
introduced in \citet{Ben-Tal1986}, and further developed in \citet{Ben-Tal2007}
that represents the specific application of the idea of Risk quadrangle
which provides a wide family of risk measures that fits the axiomatic
formalism of convex risk measures. We first introduce the following
definition. 
\begin{defn}
Let $u:\,\mathbb{R}\rightarrow[-\infty,\,\infty)$ be a closed, concave
and nondecreasing utility function with nonempty domain. The \emph{optimized
certainty equivalent (OCE) }of a random variable $X\in\mathcal{L}$
under $u$ is
\begin{equation}
S_{u}(X)=\sup_{\eta\in\mathbb{R}}\left\{ \eta+\mathbb{E}^{P}\left[u\left(X-\eta\right)\right]\right\} ,\label{OCE}
\end{equation}
\end{defn}

The OCE can be interpreted as the value obtained by an optimal allocation
between receiving a sure amount $\eta$ out of the future uncertain
amount $X$ now, and the remaining, uncertain amount $X-\eta$ later,
where the utility function $u$ effectively captures the \textquotedbl present
value\textquotedbl{} of this uncertain quantity. It turns out that
OCE measures have a dual description in terms of a convex risk measure
with a penalty function described by a type of generalized relative
entropy function called the $\phi$-divergence.
\begin{defn}
Let $\Phi$ be the class of all functions $\phi:\,\mathbb{R}\rightarrow\mathbb{R}\cup\{+\infty\}$
which are closed, convex, and have a minimum value of 0 attained at
1, and satisfy dom $\phi\subseteq\mathbb{R}_{+}$. 
\end{defn}

The framework defining the OCE in terms of a concave utility function
is derived in the context of random variables representing gains,
whereas our concern is with random variables representing losses.
To capture this difference, we will use the risk measure $\rho(X)=-S_{u}(-X)$,
where $u(t)=-\phi^{\ast}(-t)$. Note that, in this case, we have
\begin{align}
\rho(X) & =-S_{u}(-X)\nonumber \\
 & =-\sup_{\eta}\left\{ \eta+\mathbb{E}^{P}\left[u(-X-\eta)\right]\right\} \nonumber \\
 & =\inf_{\eta}\left\{ \eta-\mathbb{E}^{P}\left[u(\eta-X)\right]\right\} \nonumber \\
 & =\inf_{\eta}\left\{ \eta+\mathbb{E}^{P}\left[\phi^{\ast}(X-\eta)\right]\right\} .\label{OCE-1}
\end{align}
The convex conjugate $\phi^{\ast}$ of a function $\phi:\,\mathcal{P}\rightarrow\mathbb{R}$
is defined as a function $\phi^{\ast}:\,\mathcal{L}\rightarrow\mathbb{R}\cup\{+\infty\}$:
\[
\phi^{\ast}(X)=\rho(X)=\sup_{Q\in\mathcal{P}}\left\{ \mathbb{E}^{Q}[X]-\phi(Q)\right\} .
\]
We list choices of $\phi$-divergence function in Appendix I (Table
Examples of $\phi-$divergence functions and their convex conjugate
functions). In this paper, we have an additional assumption on $\phi^{\ast}$
throughout this paper.
\begin{assumption}
\label{Assumption 3.3} $\phi^{\ast}$ is continuous and the subdifferential
is nonempty for any element $X\in\mathcal{L}$. 
\end{assumption}

Assumption \ref{Assumption 3.3} will contribute to the development
of algorithms introduced in the latter sections. Based on \citet[Theorem 4.2.1]{Brown2006},
we can prove that formulation (\ref{OCE-1}) is equivalent to:
\[
\rho(X)=\sup_{Q\in\mathcal{P}}\left\{ \mathbb{E}^{Q}[X]-\int_{\Omega}\phi\left(\frac{dQ}{dP}\right)dP\right\} ,
\]
and therefore the penalty function $h(Q)$ in the definition of convex
risk measure (\ref{Convex risk measure}) is just the $\phi$-divergence
of $Q$ with respect to the reference measure $P$. Thus by selecting
different kinds of divergence function in OCE framework, we are able
to construct different kind of convex risk measure. Next we construct
the satisficing measure based on OCE representation of risk. Firstly,
we introduce following theorem summarizing the dual relationship between
satisficing measure with its corresponding risk measure.
\begin{thm}
\citet[Theorem 1]{Brown2009}\label{Theorem 3.3} A function $\mu:\,\mathcal{L}\rightarrow[0,\,\bar{p}]$,
where $\bar{p}\in\left\{ 1,\,\infty\right\} ,$is a satisficing measure
if and only if there exists a family risk measures $\left\{ \rho_{k}:\,k\in(0,\,\bar{p}]\right\} ,$
non-decreasing in $k$, and $\rho_{0}=-\infty$ such that 

\[
\mu(X)=\sup\left\{ k\in[0,\,\bar{p}]:\,\rho_{k}(X)\leq0\right\} .
\]
Moreover, given $\mu$, the corresponding risk measure is 

\[
\rho_{k}(X)=\inf\left\{ a:\,\mu(X+a)\geq k\right\} .
\]
\end{thm}

Theorem \ref{Theorem 3.3} shows that we could model the satisficing
measure of a random variable, given the formulation of a family of
risk measure $\rho_{k}$. To construct the satisificing measure, we
first need to define a family of risk measure $\{\rho_{\alpha}:\,\alpha\in(0,\,1]\}$.
Throughout this paper, we define a family of regret in Risk quadrangle
framework as $\mathcal{U}(X)=f(\alpha)\cdot\mathbb{E}^{P}\left[\phi^{\ast}(X-\tau)\right]$,
where $f(\alpha)$ is any function decreasing and differentiable in
$\alpha\in(0,\,1]$. For simplicity, we define $f(\alpha)=1/\alpha$
in the general model. The following is our general satisficing more
model.

\begin{equation}
\max_{\alpha\in(0,\,1]}\left\{ 1-\alpha:\,\inf_{\eta}\left\{ \eta+\frac{1}{\alpha}\mathbb{E}^{P}\left[\phi^{\ast}(X-\tau-\eta)\right]\right\} \leq0\right\} .\label{General problem-1}
\end{equation}
The interpretation of model (\ref{General problem-1}) is that we
find the maximal satisficing measure of a random outcome, by seeking
to choosing the minimal $\alpha$, so that the underlying risk measure
remain risk non-attainment. Intuitively, it computes the probability
such that the realization of the random variable does not exceeds
the risk measure $\rho_{\alpha}$. The lower the optimal value of
(\ref{General problem-1}) is, the less ``risky'' the random variable
is. In addition, if there exists a group of random outcome $X_{i}\,i=1,2,...,M$
with their target $\tau_{i},\,i=1,2,...,M$, we would rank and compare
their satisficing measure by the corresponding optimal value of the
model. Solving Problem (\ref{General problem-1}) is equivalent to
solving a sequence of convex optimization problem. We would initialize
a certain $\alpha\in[0,\,1)$, and perform binary search to minimize
$\alpha$ until the constraint of Problem (\ref{General problem-1})
violates.
\begin{example}
We illustrate in this section about the risk measure: Conditional
value-at-risk can be fitted our general framework. CVaR is the most
widely investigated coherent and law-invariant risk measure. CVaR
known also as Mean Excess Loss, Mean Shortfall, or Tail VaR, has been
widely applied because of its computational characteristics. We use
the definition of CVaR in \citet{Rockafellar2000}. For a random outcome
$X\in\mathcal{L}$, choose a specified confidence level $\alpha$
in $(0,\,1]$ , and the conjugate of $\phi$-divergence function $\phi^{*}(x)=x1_{(0,\infty)}(x)$,
the risk measure becomes
\begin{equation}
\rho_{\alpha}(x)=\inf_{\eta\in\mathbb{R}}\left\{ \eta+(1/\alpha)\text{\ensuremath{\mathbb{E}}}^{P}\left[(X-\eta)_{+}\right]\right\} .\label{CVaR formulation}
\end{equation}
Then Problem (\ref{General problem-1}) becomes:
\begin{align}
 & \text{\ensuremath{\max}}_{\alpha\in(0,\,1]}\,\left\{ 1-\alpha:\,\rho_{\alpha}(X)-\textrm{\ensuremath{\tau}}\leq0\right\} \label{CVaR-2}\\
= & \text{\ensuremath{\max}}_{\alpha\in(0,\,1]}\,\left\{ 1-\alpha:\,\inf_{\eta\in\mathbb{R}}\left\{ \text{\ensuremath{\eta}+}(1/\alpha)\mbox{\ensuremath{\mathbb{E}}}\left[\left(X-\tau-\eta\right)_{+}\right]\right\} \leq0\right\} .\nonumber 
\end{align}
Specially for CVaR, next theorem shows that Problem (\ref{CVaR-2})
possesses promising computational property, since it is equivalent
to an unconstrained convex optimization problem.
\end{example}

\begin{thm}
\label{Theorem 5.4} Problem (\ref{CVaR-2}) is equivalent to a convex
optimization problem. 
\end{thm}

\begin{proof}
Based on the arguments in \citet[Section 3.4]{Brown2009}, since CVaR
is a coherent risk measure, and noting that $\text{\ensuremath{\eta}+}\frac{1}{\alpha}\cdot\mbox{\ensuremath{\mathbb{E}}}\left[\left(X-\tau-\eta\right)_{+}\right]>0$
for all $\eta>0$, then Problem (\ref{CVaR-2}) is equivalent to
\begin{align*}
 & \text{\ensuremath{\sup}}\left\{ 1-\alpha:\,\phi_{\alpha}(X-\tau)\leq0\right\} \\
= & \sup\left\{ 1-\alpha:\,\exists\eta\leq0:\,\text{\ensuremath{\eta}+}(1/\alpha)\mbox{\ensuremath{\mathbb{E}}}\left[\left(X-\tau-\eta\right)_{+}\right]\leq0\right\} \\
= & \sup\left\{ 1-\alpha:\,\exists\eta\leq0:\,-1\text{+}(1/\alpha)\mbox{\ensuremath{\mathbb{E}}}\left[\left(-(X-\tau)/\eta+1\right)_{+}\right]\leq0\right\} \\
= & \sup\left\{ 1-\alpha:\,\exists\eta\leq0:\,1-\alpha\leq1-\mbox{\ensuremath{\mathbb{E}}}\left[\left(-(X-\tau)/\eta+1\right)_{+}\right]\right\} \\
= & \sup\left\{ 1-\mathbb{E}\left[\left(\mbox{\ensuremath{\alpha}}(X-\tau)+1\right)_{+}\right]\right\} ,\\
= & \sup_{0<\alpha\leq1}\left\{ \mathbb{E}\left(\min\left\{ -\alpha\left(X-\tau\right),\,1\right\} \right)\right\} ,
\end{align*}
and it is easy to show that the problem is a convex optimization problem
based on operation preserving convexity in \citet[Section 3.2]{Boyd2004}
that the minimum of a piecewise linear function is a concave function.
The interpretation of this model is to find the optimal $\alpha$
that maximize the expected utility of a concave function. Thus Problem
(\ref{CVaR-2}) can be solved as a single convex optimization problem
rather than a sequence of optimization problem.
\end{proof}

\section{Batch learning}

In this section, we study computationally tractable techniques for
risk ranking based on Problem (\ref{General problem-1}). One natural
approach is Batch learning i.e., Sample average approximation (SAA).The
SAA principle is very general, having been applied to settings including
chance constraints, stochastic-dominance constraints and complementary
constraints problems. Since Problem (\ref{General problem-1}) is
a expected value constrained problem, the constraints must also be
evaluated using simulation. Batch learning approach optimization is
appropriate in the cases when: (1) We have complete information on
the probability distribution of $X$, and we are able to randomly
extract large samples from that distribution. Approximation technique
are applied to tackle the difficulty in computing the expectations
in the constraints; (2) The optimization could be conducted based
on samples $d_{n},\,n=1,2,...,N$ draw from an unknown distribution
for each item, and when $N$ is large enough. For both cases, Sample
average approximation (SAA) can be introduced to make the optimization
problem tractable, and the almost sure convergence, convergence rate
of feasible region and optimal solution validation of SAA have been
studied in \citet{Wang2008,Hu2012,Homem-De-Mello2003,Kim2015}. The
novel contribution in this section lies in that we would derive the
quality of satisficing ranking with sample complexity based on the
convergence rate and optimal solution validation results of SAA of
Problem (\ref{General problem-1}). i.e., figure out the probability,
that the ranking by SAA problem is equivalent to the true underlying
risk ranking, with the sample size required for each random outcomes.
One promising advantage of model \ref{General problem-1} is that
the optimal value equals to one minus the optimal solution. Such advantage
will be extremely helpful to show the validation of ranking by the
convergence rate results in SAA. We proceed to validating SAA for
a risk ranking system with $I$ different random variables with their
target $\tau_{i},\,i=1,\,2,...,\,I$. Suppose their underlying optimal
risk assessment values are $\alpha_{i},\,i=1,\,2,...,\,I$. 

\subsection{Algorithm }

In this section, we proceed the idea of SAA on Problem (\ref{General problem-1}).
Choosing $N$ samples from underlying probability distribution of
$X$ as $d_{n},\,n=1,2,...,N$, we have SAA problem
\begin{align}
\max_{\alpha\in(0,\,1]}\left\{ 1-\alpha:\,\inf_{\eta}\left\{ \eta+\frac{1}{\alpha N}\sum_{n=1}^{N}\left[\phi^{\ast}(d_{n}-\tau-\eta)\right]\right\} \leq0\right\} ,\label{SAA}
\end{align}
where the expected value in Problem (\ref{General problem-1}) is
replaced by sample average approximator. Following Binary search algorithm
(Algorithm 1) provides the computational methods for Problem (\ref{SAA}).

\begin{algorithm}
\textbf{Input: }Tolerance level $\epsilon>0$; Conjugate of divergence
function $\phi^{\star}$ ; Batch of samples $d_{n},\,n=1,2,...,N,$

\textbf{Output:} Optimal $\alpha^{\ast}$;

\textbf{Initialization: }$\alpha_{\min}\leftarrow0$, and $\alpha_{\max}\leftarrow1;$ 

\textbf{while} $\left(\alpha_{\max}-\alpha_{\min}>\epsilon\right)$
\textbf{do}

$\qquad$Compute: $\alpha=\left(\alpha_{\max}+\alpha_{\min}\right)/2$;
Solve the subproblem and check the feasibility:
\begin{align*}
\inf_{\eta}\left\{ \eta+\frac{1}{\alpha N}\sum_{n=1}^{N}\left[\phi^{\ast}(d_{n}-\tau-\eta)\right]\right\} ,
\end{align*}

$\qquad$and let $m$ denotes its optimal value.

$\qquad$\textbf{If} $m\leq\tau$, \textbf{then} $\alpha_{\max}\leftarrow\alpha$;

$\qquad$\textbf{else} $\alpha_{\min}\leftarrow\alpha$;

\textbf{end}

\textbf{return} optimal $\alpha^{\ast}=\alpha_{\max}$;

\caption{Binary search algorithm for SAA problem}
\end{algorithm}

\subsection{Main results}

In this section, we argue the validation of SAA in terms of risk ranking
with sample size. The arguments are based on the convergence rate
of approximated feasible region to that of the initial risk-constrained
problem by Large Deviation analysis proposed in \citet{Wang2008},
and lower and upper bound of optimal solution by Central limit theorem
and Law of large number. To start, define function $G:\,[0,\,1]\times\mathcal{L}\rightarrow\mathbb{R}$:
\[
G(\alpha,\,\eta,\,X)=\eta+\frac{1}{\alpha}\left[\phi^{\ast}(X-\eta)\right],
\]
and we have following assumptions on function $G$.
\begin{assumption}
\label{Assumption 4.1} The following assumption will be required
:\\

(C1) For any $X\in\mathcal{L}$ there exists an integrable function
$\psi:\,\mathcal{L}\rightarrow\mathbb{R}_{+}$ such that
\[
|G(\alpha_{1},\,\eta_{1},\,X)-G(\alpha_{2},\,\eta_{2},\,X)|\leq\phi(X)\|\alpha_{1}-\alpha_{2}\|+\psi(X)\|\eta_{1}-\eta_{2}\|,
\]
for all $\alpha_{1},\,\alpha_{2}\in(0,\,1]$, and $\eta_{1},\,\eta_{2}\in\mathbb{R}$.
Denote $\Psi:=\mathbb{E}\left[\psi(X)\right]$, and $\Phi:=\mathbb{E}\left[\phi(X)\right]$.\\

(C2) The Moment generating function $M_{\psi}(\cdot)$ of $\psi(X)$
and $M_{\phi}(\cdot)$ of $\phi(X)$ are finite in a neighborhood
of zero.\\

(C3) For any $\alpha\in(0,\,1]$, and $\eta\in\mathbb{R}$, the moment
generating function of $M_{\text{\ensuremath{\alpha},}\,\eta}(\cdot)$
of \textup{$G(\alpha,\,\eta,\,X)-\mathbb{E}\left[G(\alpha,\,\eta,\,X)\right]$
is finite around zero.}
\end{assumption}

\begin{lem}
\label{Lemma 4.2} Suppose assumption \ref{Assumption 4.1} hold.
There exists a closed interval $E\subset\mathbb{R}$ such that problem
\ref{General problem-1} is equivalent to
\[
\max_{0\leq\alpha\leq1,\,\eta\in E}\left\{ 1-\alpha:\,g(\alpha):=\mathbb{E}^{P}\left[G\left(\alpha,\,\eta,\,X\right)\right]\leq\tau\right\} ,
\]
\end{lem}

Next Theorem illustrates obtaining and validating candidate solutions
using SAA method and figure out the lower and upper bound for the
optimal value of SAA problem.
\begin{prop}
\label{Proposition 4.3} The upper and lower bound of SAA problem
(\ref{SAA}) can be derived by solving following Lagrangian
\[
\max_{0\leq\alpha\leq1}\left\{ 1-\alpha+\left\langle \pi,\,g_{N}(\alpha)-\tau\right\rangle \right\} ,
\]
and let $(\tilde{\alpha},\,\tilde{\pi})$ denote its optimal primal-dual
pair (In order to solve Problem (\ref{SAA}) efficiently, it is very
possible that $\tilde{\alpha}$ is infeasible to original problem.
Therefore a smaller right-hand-side $\tilde{\tau}\leq\tau$ can improve
the chance that $\tilde{\alpha}$ is feasible). We can derive the
following bounding results.

(i) Upper bound: Let $\{d_{1},\,d_{2},...,\,d_{N_{q}}\}$be another
sample obtained by resampling technique with size $N_{q}$ that $N_{q}\gg N.$
Compute
\[
\tilde{q}:=\sup_{\eta\in\mathbb{R}}\left\{ N_{q}^{-1}\sum_{n=1}^{N_{q}}\left[\eta+\frac{1}{\tilde{\alpha}}\phi^{*}\left(d_{n}-\eta\right)\right]\right\} ,
\]
and
\[
S_{\tilde{q}}^{2}:=\sup_{\eta\in\mathbb{R}}\left\{ N_{q}^{-1}\left(N_{q}-1\right)^{-1}\sum_{n=1}^{N_{q}}\left[\eta+\frac{1}{\tilde{\alpha}}\phi^{*}\left(d_{n}-\eta\right)-\tilde{q}\right]^{2}\right\} .
\]
Define $z_{\delta}$ by $\mathbb{P}(Z\leq z_{\delta})=1-\delta,$
where $Z$ is a standard normal random variable and $\delta\in[0,\,1]$,
by computing $z_{\delta}=\frac{\tau-\tilde{q}}{S_{\tilde{q}}}$, if
$z_{\delta}$ is big enough, we can conclude that $1-\tilde{\alpha}$
is an upper bound for original problem with probability $1-\delta$;
otherwise, we decrease $\tilde{\tau}$ and solve the whole problem
until it terminates, 

(ii) Lower bound: Generate $M_{l}$ independent group of samples each
of size $N_{l}$,i.e., $\{d_{1}^{m},\,d_{2}^{m},...,\,d_{N_{l}}^{m}\}$for
$m=1,\,2,...,\,M_{l}$. For each sample group, solve the SAA problem:
\[
\hat{l}^{m}:=\max_{0\leq\alpha\leq1}\left\{ 1-\alpha+\tilde{\pi}\left[\sup_{\eta\in\mathbb{R}}N_{l}^{-1}\sum_{n=1}^{N_{l}}\left\{ \eta+\frac{1}{\alpha}\cdot\phi\left(d_{n}^{m}-\eta\right)\right\} -\tau\right]\right\} .
\]
Compute the lower bound estimator $\tilde{l}$ and its variance $S_{\tilde{l}}^{2}$
as follows
\[
\tilde{l}:=\frac{1}{M_{l}}\sum_{m=1}^{M_{l}}\hat{l}^{m},
\]
and
\[
S_{\tilde{l}}^{2}:=\frac{1}{M_{l}(M_{l}-1)}\sum_{m=1}^{M_{l}}\left(\hat{l}^{m}-\tilde{l}\right)^{2}.
\]
Then $\tilde{l}_{L}=\widetilde{l}-z_{\gamma/2}S_{\tilde{l}}$ is a
lower bound on the true optimal value of Problem (\ref{General problem-1})
with confidence level $(1-\gamma)$, where $z_{\gamma/2}$ is the
$\gamma/2$ quantile value of standard normal distribution. 
\end{prop}

Next assumption derives the sensitivity of ranking and upper bound
on the SAA value $\hat{l}^{m}$. 
\begin{assumption}
\label{Assumption 4.4} The following assumptions will be required:

(i) For all the optimal risk assessment value $\alpha_{i},\,i=1,\,2,...,\,I$.
We rank them from low to high that $0\leq1-\alpha^{(1)}<1-\alpha^{(2)}<\cdot\cdot\cdot<1-\alpha^{(I)}\leq1$,
and $\alpha^{(i)}$ denotes the $i$th largest risk assessment value,
there exists a small positive number $c$ that:
\[
\text{\ensuremath{\min}}_{i}\Bigl|\alpha^{(i)}-\alpha^{(i-1)}\Bigr|\geq c.
\]

(ii) There exists a large positive value $M$, so that $\hat{l}^{m}\leq M$
for $m=1,\,2,...,\,M_{l}$

(iii) There exists a positive number $C$ so that, $\tilde{\pi}\leq C$
uniformly, and $\epsilon_{0}>0$ that $\|\tau-\tilde{\tau}\|_{2}\leq\epsilon_{0}$.
\end{assumption}

\begin{prop}
\label{Proposition 4.5} If assumption \ref{Assumption 4.4} (ii)
holds, by Popoviciu inequality, we have $S_{\tilde{l}}^{2}\leq\frac{M^{2}}{4}$.
\end{prop}

Given $v>0$, build a finite set $U_{v}$ of $U$ such that for any
$x\in X$, there exists $x^{\prime}\in X_{v}$ satisfying $\|x-x^{\prime}\|\leq v$.
Denoting by $D_{E}$ the diameter of set $X$, i.e., $D_{E}=\max_{x_{1},x_{2}\in X}\|x_{1}-x_{2}\|$,
then such set $X_{v}$ can be construct with $|X_{v}|\leq(D_{E}/v)$,
the finite set $X_{v}$ is called $v$-net of set $X$, then we have
the following main theorem of risk ranking that studies the relationship
between the sample size and the validation probability of ranking
by SAA problem to the original problem.
\begin{thm}
\label{Theorem 4.6} Based on assumptions \ref{Assumption 4.1} and
\ref{Assumption 4.4}, propositions \ref{Proposition 4.3} and \ref{Proposition 4.5},
and \ref{Lemma 4.2}z, then there exists $0<\epsilon\leq c$, that
the ranking for $I$ items from SAA method is the same as the ranking
from optimal value of Problem (\ref{General problem-1}) with the
probability $\mathbb{P}$ that:
\[
\mathbb{P}\geq\left[\left(1-\beta\right)\cdot\left(1-\gamma\right)\cdot\left(1-\delta\right)\right]^{I},
\]
with sample size for each item
\begin{align*}
N & \geq\max\left\{ \frac{8\sigma^{2}}{\epsilon^{2}}\log\left[\frac{2}{\beta}\left(2+\frac{D_{E}}{v_{1}^{2}}\right)\right],\,\frac{8\sigma^{2}C^{2}}{(\epsilon-C\epsilon_{0})^{2}}\log\left[\frac{2}{\beta}\left(2+\frac{D_{E}}{v_{2}^{2}}\right)\right],\right\} \\
 & +\frac{8\sigma^{2}C^{2}}{(c-\epsilon-z_{\gamma/2}\cdot\frac{M^{2}}{4})^{2}}\log\left[\frac{2}{\beta}\left(2+\frac{D_{E}}{v_{3}^{2}}\right)\right]\times M_{l},
\end{align*}
and totally $I\times N$ samples , where
\[
v_{1}:=\left\{ 4(\Psi+\Phi)/\epsilon+2\right\} ^{-1},
\]
\[
v_{2}:=\left\{ 4(\Psi+\Phi)R/(\epsilon-C\epsilon_{0})+2\right\} ^{-1},
\]
\[
v_{3}:=\left\{ 4(\Psi+\Phi)/(c-\epsilon-z_{\gamma/2}\cdot\frac{M^{2}}{4})+2\right\} ^{-1},
\]
and
\[
\sigma^{2}:=\text{\ensuremath{\max}}_{\alpha\in(0,1],\,\eta\in E}\left\{ \textrm{VaR}\left[\phi(X)\right],\,\textrm{VaR}\left[\psi(X)\right],\,\textrm{VaR}\left[G(\alpha,\,\eta,\,X)-\mathbb{E}\left[G(\alpha,\,\eta,\,X)\right]\right]\right\} .
\]
\end{thm}

\begin{proof}
See Appendix II
\end{proof}
For next theorem, we illustrate the problem in a reverse way by studying
given a certain sample size $N=N_{1}+N_{2}\times M_{l}$, where $N_{1},\,N_{2}>0$,
what are the probability that the ranking system valid correlated
with $N$ .
\begin{thm}
\label{Theorem 4.7} Based on results in Theorem \ref{Theorem 4.6},
then given sample size $N=N_{1}+N_{2}\times M_{l}$, where $N_{1},\,N_{2}>0$,
there exists $0<\epsilon\leq c$, so that the ranking for $I$ items
from SAA method is the same as the ranking from optimal value of Problem
(\ref{General problem-1}) with the probability $\mathbb{P}$ that:
\begin{align*}
\mathbb{P} & \geq\min\Biggl\{1-2\left[2+\frac{D_{E}}{v_{1}^{2}}\right]\exp\left(-\frac{N_{1}\epsilon^{2}}{8\sigma^{2}}\right)\\
 & 1-2\left[2+\frac{D_{E}}{v_{2}^{2}}\right]\exp\left(-\frac{N_{1}(\epsilon-R\epsilon_{0})^{2}}{8\sigma^{2}R^{2}}\right)\Biggr\}^{I}\\
 & \times\left\{ 1-2\left[2+\frac{D_{E}}{v_{3}^{2}}\right]\exp\left\{ \frac{-N_{2}\cdot(c-\epsilon-z_{\gamma/2}\cdot\frac{M^{2}}{4})^{2}}{8\sigma^{2}}\right\} \right\} ^{I},
\end{align*}
where $v_{1},\,v_{2},\,v_{3}$ and $\sigma$ are the same in Theorem
\ref{Theorem 4.6}.
\end{thm}

\section{Online learning with stochastic approximation}

In the above sections, we introduced sample average approximation
as a data-driven technique to tackle our risk ranking problem. SAA
is the appropriate computational methods based on batch learning,
and we argue its large-scale sample performance for risk ranking problem
in the last section. While, for the operations of systems in real
world, little data might be collected at the beginning but massive
new data become accessible sequentially. It is important to develop
some computational methods to dynamically adjust the risk-assessment
and ranking results while learning new random outcome from an unknown
probability distribution at each step. i.e., real-time risk assessment.
A well-understood, general-purpose method for solving stochastic optimization
problems, alternative to using the SAA principle, is called stochastic
approximation (SA) proposed in \citet{Kushner2003,Kiefer1952,Robbins1951}.
Classic stochastic approximation algorithm is in a recursion formulation
that processes a sequence of data to estimate expected value in an
online optimization way. 

For solving general unconstrained optimization problem, SA wins SAA
asymptotically on the computational effort required to obtain solutions
of a given quality based on the results from \citet{Kim2015}. Another
computational advantage by using SA enable us to derive the quality
of risk assessment result with the number of iterations, while for
SAA we have to apply statistical methods to figure out the lower and
upper bound for the optimal value summarized in Proposition \ref{Proposition 4.3}
to infer the quality of ranking results. However, risk ranking by
SAA dominates SA in terms of the generality. First, for SAA, we can
study whether the approximated risk ranking is exactly the same as
the true ranking results, while, for SA, we have to construct a loss
function as a relax metric for measuring the quality of ranking. Second,
the quality of ranking analysis for SA is developed based on the fact
that the solution gap of SA follows certain probability distribution
(e.g., Gumbel, Normal, Gamma). So the choice of the distribution will
largely affect the theoretical analysis result. SAA provides the probability
bounds based on statistical methods, which has much higher reliability. 

In the following subsections, we will derive the stochastic approximation
algorithm for Problem (\ref{General problem-1}), and develop main
results and analysis on risk ranking.

\subsection{Algorithm}

In this section, we will seek to develop the algorithm to solve Problem
(\ref{General problem-1}) in stochastic approximation way. The form
of stochastic optimization problem immediately leads to an online
gradient/subgradient descent algorithm, and similar algorithms have
already been investigated in following papers (\citealt{Mahdavi2012,Jiang2015,Duchi2011}).
These papers suggest the use of stochastic approximation or stochastic
gradient/subgradient descent algorithms as a crucial step to estimate
the expectation. As our first step, we transform Problem (\ref{General problem-1})
into an unconstrained problem: an appropriate saddle-point problem
which leads to an online primal-dual algorithm, and let Lagrangian
for Problem (\ref{General problem-1}) to be $L:\,(0,\,1]\times E\times(-\infty,\,0]\rightarrow\mathbb{R}$
given by
\[
L(\alpha,\,\eta,\,\lambda)=\mathbb{E}^{P}\left\{ 1-\alpha+\lambda\left[\eta+\frac{1}{\alpha}\left[\phi^{\ast}(X-\tau-\eta)\right]\right]\right\} ,
\]
Since function $\phi^{\ast}$ is a convex function, the Lagrangian
is concave in $(\alpha,\,\eta)$ and convex (linear) in $\lambda$,
then by the duality theory of convex optimization problem, finding
the optimal solution of Problem (\ref{General problem-1}) is equivalent
to solving the following saddle problem
\begin{equation}
\max_{(\alpha,\,\eta)\in(0,1]\times E}\min_{\lambda\leq0}L(\alpha,\,\eta,\,\lambda),\label{Saddle}
\end{equation}
which naturally yields the solution algorithm to be an online primal
dual-algorithm with stochastic subgradient method. We define the realization
of the Lagrangian would be
\begin{equation}
\left[L(\alpha,\,\eta,\,\lambda)\right](d_{t})=1-\alpha+\lambda\left[\eta+\frac{1}{\alpha}\left[\phi^{\ast}(d_{t}-\tau-\eta)\right]\right],\label{Lagrangian}
\end{equation}
where $d_{t}\in X$, $t=1,2,....,T$ are the sequential outcomes of
a variable in $T$ iterations.

\begin{algorithm}
\textbf{Input}: total iterations $T$, time step $t=0.$ 

\textbf{Step 1}. Set $(\alpha_{0},\,\eta_{0},\,\lambda_{0})\in(0,\,1]\times E\times(-\infty,\,0]$,
and $R_{0}=0$.

\textbf{for} $t=1,2,...,T$ \textbf{do}

\textbf{Step 2}. Obtain $d_{t}$ on $X$ and observe saddle function
$(\alpha,\,\eta,\,\lambda)\rightarrow\left[L(\alpha,\,\eta,\,\lambda)\right](d_{t}).$

\textbf{Step 3}. Compute sub-gradient $g_{t}\in\partial_{(\alpha,\eta)}\left[L(\alpha_{t-1},\,\eta_{t-1},\,\lambda_{t-1})\right](d_{t})$
and $\bigtriangledown_{\lambda}\left[L(\alpha_{t-1},\,\eta_{t-1},\,\lambda_{t-1})\right](d_{t})$.

\textbf{Step 4}. Update:
\[
(\alpha_{t},\,\eta_{t})=\text{\ensuremath{\prod}}_{(0,\,1]\times E}\left((\alpha_{t-1},\,\eta_{t-1})+\frac{g_{t}}{t}\right),
\]
\[
\lambda_{t}=\text{\ensuremath{\prod}}_{(-\infty,\,0]}\left(\lambda_{t-1}-\frac{1}{t}\bigtriangledown_{\lambda}\left[L(\alpha_{t-1},\,\eta_{t-1},\,\lambda_{t-1})\right](d_{t})\right),
\]
\[
R_{t}=R_{t-1}-\frac{1}{t}\left(R_{t-1}-\left[L(\alpha_{t-1},\,\eta_{t-1},\,\lambda_{t-1})\right](d_{t})\right).
\]

\textbf{end for}

\textbf{Step 5}. Return $R_{T}$

\caption{Online primal-dual sub-gradient descent algorithm}
\end{algorithm}
In the above algorithm $R_{T}$ will become the estimation for the
optimal value of Problem (\ref{Saddle}), and in next section, we
proceed to validating SA for a risk ranking system with $I$ different
random variables $X_{i}.\,i=1,\,2,...,\,I$ with their target $\tau_{i},\,i=1,\,2,...,\,I$
by deriving the convergence rate of $R_{T}$ to the optimal value
of Problem \ref{Saddle} $R^{\ast}$ i.e., Optimal gap with respect
to the number of iterations.

\subsection{Main results}

In this section, we seek to prove the validation of risk ranking produced
by online learning i.e., stochastic approximation methods in Algorithm
2, and find the minimal iteration number $T$ (i.e., the minimal sample
size required for each item) required to satisfy certain accuracy
of ranking. For risk ranking with SAA, we focus on figuring out the
minimal required sample size to ensure the approximated risk ranking
is exactly the same as the true ranking results. In this section,
we would relax such condition by defining a new metric for measuring
the quality of ranking as follows.
\begin{defn}
\label{Definition 5.1} Given $R_{T}^{i},\,i=1,\,2,...,\,I$ denote
the results from Algorithm 2, and their corresponding true optimal
solution of Problem (\ref{General problem-1}) as $R_{i},\,i=1,\,2,...,\,I$.
We define the measures of the quality of the ranking as a loss function
$\mathcal{E}$:
\[
\mathcal{E}=\sum_{i}\sum_{j}I_{\left\{ \left(R_{T}^{i}-R_{T}^{j}\right)\left(R_{i}-R_{j}\right)<0\right\} },\:i,\,j=1,2,...,I,\,i\neq j.
\]
The interpretation of loss function is to sum up the total inversion
number in a ranking system based on the idea of evaluation metrics
of ranking algorithms\textcolor{black}{{} in \citet{Chen2012,Frey2007,Kriukova2016,Zhang2012a,Chen2013}.}
If the closer the ranking derived by online algorithm to the underlying
true ranking is, the few the inversion number will produce, then the
value of loss function $\mathcal{E}$ will become lower. Given $e>0$,
we plan to figure out the relationship between the probability that
loss $\mathcal{E}$ is bounded by $e$ (i.e., $\mathbb{P}\left(\mathcal{E}\leq e\right)$)
with total iteration number $T$. It is anticipated that such probability
will increase with $T$ and $e$. For the first step to study the
ranking quality, we seek to establish the almost sure convergence
results and convergence rate of Algorithm 2 based on following assumptions. 
\end{defn}

\begin{assumption}
\label{Assumption 5.2} (i) (Bounded derivative) For all $(\alpha,\,\eta,\,\lambda)\in(0,\,1]\times E\times(-\infty,\,0]$,
Let $f(\alpha,\,\eta,\,\lambda)$ denote
\[
f(\alpha,\,\eta,\,\lambda)=\max\left\{ \|\bigtriangledown_{\alpha}L(\alpha,\,\eta,\,\lambda)(d_{t})\|_{2}^{2},\,\|\bigtriangledown_{\eta}L(\alpha,\,\eta,\,\lambda)(d_{t})\|_{2}^{2},\,\|\bigtriangledown_{\lambda}L(\alpha,\,\eta,\,\lambda)(d_{t})\|_{2}^{2}\right\} ,
\]
and
\[
g(\alpha,\,\eta,\,\lambda)=\mathbb{E}\left[f(\alpha,\,\eta,\,\lambda)\right],
\]
then, there exists $M>0$ that
\[
g(\alpha,\,\eta,\,\lambda)\leq M^{2},
\]

(ii) (Strong convexity) For $\alpha_{1},\alpha_{2}\in(0,\,1]$, and
for all $\eta_{1},\eta_{2}\in E$ and $\lambda_{1},\lambda_{2}\leq0$,
there exists a positive value $L_{g}$ that
\[
|g(\alpha_{1},\,\eta_{1},\,\lambda_{1})-g(\alpha_{2},\,\eta_{2},\,\lambda_{2})|\leq L_{g}\left(\|\alpha_{1}-\alpha_{2}\|_{2}^{2}+\|\eta_{1}-\eta_{2}\|_{2}^{2}+\|\lambda_{1}-\lambda_{2}\|_{2}^{2}\right).
\]

(iii) (Lipschitz continuity) For $\alpha_{1},\alpha_{2}\in(0,\,1]$,
and for all $\eta_{1},\eta_{2}\in E$ and $\lambda_{1},\lambda_{2}\leq0$,
there exists a positive value $L_{\Phi}$ that
\[
\mathbb{E}\left(\left[L(\alpha_{t-1},\,\eta_{t-1},\,\lambda_{t-1})\right](d_{t})-\left[L(\alpha^{\ast},\,\eta^{\ast},\,\lambda^{\ast})\right](d_{t})\right)\leq L_{\Phi}\left(\|\alpha_{t}-\alpha^{\ast}\|_{2}^{2}+\|\eta_{t}-\eta^{\ast}\|_{2}^{2}+\|\lambda_{t}-\lambda^{\ast}\|_{2}^{2}\right).
\]
\end{assumption}

Based on above Assumption \ref{Assumption 5.2}, we have following
Theorem on the deriving optimality gap of stochastic approximation
algorithm.
\begin{thm}
\label{Theorem 5.3} By Assumption \ref{Assumption 5.2}, given any
$\kappa_{0}>0$, the solution $R_{T}$ derived by Algorithm 2 satisfies
\[
\mathbb{E}\left[\|R_{T}-R^{\ast}\|_{2}^{2}\right]\leq\kappa^{\prime}/T,
\]
where
\[
\kappa^{\prime}:=\max\left\{ B-2-\frac{3\kappa L_{\Phi}}{\kappa_{0}}/(2\kappa_{0}-1)^{-1},\,\|R_{0}-R^{\ast}\|_{2}^{2}\right\} ,
\]
and
\[
\kappa:=\max\left\{ M^{2}/(2L_{g}-1)^{-1},\,\|\alpha_{0}-\alpha^{\ast}\|_{2}^{2},\,\|\eta_{0}-\eta^{\ast}\|_{2}^{2},\,\|\lambda_{0}-\lambda^{\ast}\|_{2}^{2}\right\} ,
\]
given the following bounds
\[
\mathbb{E}\left[R_{t-1}-\left[L(\alpha_{t-1},\,\eta_{t-1},\,\lambda_{t-1})\right](d_{t})\right]^{2}\leq B,
\]
and
\[
M^{2}:=\sup_{(\alpha,\,\eta)\in(0,1],\,\lambda\leq0}\mathbb{E}\left[\|g(\alpha,\,\eta,\,\lambda)\|_{2}^{2}\right].
\]
\end{thm}

\begin{proof}
See Appendix III
\end{proof}
The result of Theorem \ref{Theorem 5.3}, we can conclude that the
convergence rate for $R_{T}$ is $O(1/T)$ and as $T$ goes to infinity,
we can claimed the almost sure convergence of our algorithm. 

Next we seek to model the probability distribution for the solution
gap $R_{T}^{i}-R_{i}$ of all the random outcomes. Modern online ranking
system like most currently used Elo variants system use a logistic
distribution rather than Gaussian because it represents the distribution
of maxima relates to extreme value theory (\citealt{Weng2011}), since
here we would model the upper bound of the solution gap $\kappa^{\prime}/T$,
and, in this section, we use Gumbel distribution for modeling. The
cumulative distribution function of Gumbel distribution is
\begin{equation}
F(x;\,\xi,\,\beta)=\exp(-\exp(-(x-\xi)/\beta)),\label{Gumbel}
\end{equation}
where $\xi$ is the location parameter and $\beta>0$ is the scale
parameter. we assume in our paper that \textbf{$\beta=1$}. The main
theorem is listed as follows.
\begin{thm}
\label{Theorem 5.4-1} For online optimization problem with general
convex risk measure constraint, given total iteration budget $T$,
by Assumption \ref{Assumption 5.2}, and Theorem \ref{Theorem 5.3},
and given $e>0$, we have:
\[
\mathcal{\mathbb{P}}\left\{ \mathcal{E}\leq e\right\} =\sum_{i=0}^{e}\left\{ \binom{C_{I}^{2}}{i}\cdot(1-p)^{i}\cdot p^{C_{I}^{2}-i}\right\} ,
\]
where $p=1-\exp(-\exp(-(c-\log(-\log(1))+\kappa^{\prime}/T))$, and
$C_{I}^{2}=\binom{I}{2}$ . $c$ is defined in Assumption \ref{Assumption 4.4}
(i) and $\kappa^{\prime}$ has the same definition as in Theorem \ref{Theorem 5.3}.
\end{thm}

\begin{proof}
Based on the equation of Gumbel distribution (\ref{Gumbel}), we can
compute the equation for $x$ that
\[
\exp(-\exp(-(\kappa^{\prime}/T-x))=1,
\]
 and solve that $x=\log(-\log(1))+\kappa^{\prime}/T$. By Definition
\ref{Definition 5.1}, we can derive the probability that there exists
an inversion in ranking is $p$,
\begin{align*}
p & =1-F(c;\,\log(-\log(1))+\kappa^{\prime}/T,\,1)\\
 & =1-\exp(-\exp(-(c-\log(-\log(1))+\kappa^{\prime}/T)),
\end{align*}
then, given any inversion bound $e$, by the basic theories of permutation
and combination, we have
\[
\mathcal{\mathbb{P}}\left\{ \mathcal{E}=e\right\} =\binom{C_{I}^{2}}{e}\cdot(1-p)^{e}\cdot p^{C_{I}^{2}-e}.
\]
Thus we can prove the desired result.
\end{proof}

\section{Numerical experiments}

\textcolor{black}{In this section, we would apply and validate our
methodologies by risk assessment of Emerging organic compounds (EOCs)
in Singapore waterbody as a case study. The contamination of the urban
water cycle with a wide array of EOCs increases with urbanization
and population density, specially in megacities like Singapore. In
\citet{Pal2014}, Environmental contamination of EOCs has been reviewed
from several perspectives, including developments in analytical techniques,
occurrence of EOCs in surface waters, ground water, sludge and drinking
water, and toxic effect and risk assessment along with regulatory
implications. The main entry points of EOCs into the urban water cycle
include households, hospitals, construction, landscaping, transportation,
commerce, industrial scale animal feeding operations, dairy farms,
and manufacturing. Additional sources include leaking sewer lines,
landfills and inappropriately disposed wastes. Representative compound
classes include hormones, antibiotics, surfactants, endocrine disruptors,
human and veterinary pharmaceuticals, X-ray contrastmedia, pesticides
and metabolites, disinfection-by- products, algal toxins and taste-and-odor
compounds. EOCs comprise recently developed industrial compounds that
have been newly introduced to the environment; compounds that have
been prevalent for some time but are only now being routinely detected
owing to improved detection techniques; and compounds that have been
prevalent for a long time but have only recently been shown to have
harmful eco-toxicological effects. Since experts have little knowledge
on the chemical properties of EOCs, it is important to assess their
risk on a mathematical point of view based on data. There exist varieties
of sources and classes of EOCs and new EOCs continue to be discovered,
and it is necessary to develop real-time computational framework to
assess and rank the ``risk'' for EOCs on the current lists. The
following table lists thirteen the most common discovered EOCs in
Singapore.}

\begin{table}
\begin{centering}
\begin{tabular}{|c|c|}
\hline 
Contaminant  & Lowest PNEC (ng/L) \tabularnewline
\hline 
\hline 
Caffeine & 5,200 \tabularnewline
\hline 
Salicylic acid  & 167,000 \tabularnewline
\hline 
Acetaminophen  & 1,400 \tabularnewline
\hline 
Crotamiton  & 21,000 \tabularnewline
\hline 
Sulpiride  & 100,000 \tabularnewline
\hline 
Chloramphenicol  & 1,600 \tabularnewline
\hline 
Naproxen  & 5,200 \tabularnewline
\hline 
Estrone  & 18 \tabularnewline
\hline 
BPA  & 60 \tabularnewline
\hline 
DEET  & 5200 \tabularnewline
\hline 
Triclosan  & 100 \tabularnewline
\hline 
Benzophenone-3  & 6,000 \tabularnewline
\hline 
Fipronil  & 130 \tabularnewline
\hline 
\end{tabular}
\par\end{centering}
\caption{EOCs and corresponding PNEC in Singapore}

\end{table}
In Table 1, PNEC is called Predicted minimum noeffect concentration.
In \citet{Pal2014}, the PNEC is defined as the concentration below
which unacceptable or harmful effects on organisms are unlikely to
occur, which served as a benchmark or target to identify the ``risk''
of one EOC. The aims of this numerical experiment are to perform our
batch learning and online learning computational framework for risk
assessment and prioritizing on EOCs, and to validate the computational
property of our algorithms.

\subsection{Batch learning}

In this numerical experiment, We seek to show that probability that
the validation of optimal solution derived by Sample average approximation
will increase with the sample size. Choose the conjugate of $\phi$-divergence
function $g_{\alpha}^{*}(x)=\frac{1}{\alpha}x1_{(0,\infty)}(x)$,
then we construct the risk measure model with Conditional value-at-risk
as:
\begin{align*}
\min_{\alpha\in(0,\,1],\,\eta\in\mathbb{R}}\, & \alpha\\
\textrm{s.t.} & \inf_{\eta\in\mathbb{R}}\left\{ \eta+\frac{1}{1-\alpha}\mathbb{E}^{P}\left[X-\eta\right]_{+}\right\} \leq\tau.
\end{align*}
We set following parameters, suppose $X\sim N(100,\,50)$, and Target
$\tau=120$. The following Table 1 shows the relationship between
optimal solution validation with sample size. This experiment is run
on Desktop Dell Optiplex by by Matlab 2013a and Cplex 12.5. It can
be observed in Table 2 that, as the sample size grows, the LB$_{0.95}$
is increasing and gap between lower and upper bound is shrinking based
on the metric $\frac{\textrm{UB}-\textrm{L\ensuremath{B_{0.95}}}}{|\textrm{UB}|}$(\%),
and remain stable for the metric $\frac{\textrm{UB-Opt.Obj}}{\textrm{|Opt.Obj|}}$(\%),
which reflects that the quality of optimal solution is increasing
with the sample size.

\begin{table}
\begin{centering}
\begin{tabular}{|c|c|c|c|l|c|c|}
\hline 
$\alpha$ & $N=50$ & $N=100$ & $N=250$ & $N=500$ & $N=1000$ & $N=5000$\tabularnewline
\hline 
\hline 
Opt.Obj & $\textrm{0.7368}$ & $\textrm{0.7429}$ & $\textrm{0.7371}$ & 0.7659 & 0.7686 & 0.7671\tabularnewline
\hline 
Time (s) & $\textrm{0.239}$ & $\textrm{0.244}$ & $\textrm{0.301}$ & 0.253 & 0.276 & 0.709\tabularnewline
\hline 
UB & $\textrm{0.7990}$ & $\textrm{0.8141}$ & 0.8040 & 0.8392 & 0.8342 & 0.8392\tabularnewline
\hline 
$\mathbb{P}\left\{ \textrm{UB}\right\} $(\%) & $\textrm{100}$\% & $\textrm{100}$\% & 100\% & 100\% & 100\% & 100\%\tabularnewline
\hline 
LB$_{0.95}$ & $\textrm{0.6756}$ & $\textrm{0.6988}$ & 0.6867 & 0.7388 & 0.7449 & 0.7566\tabularnewline
\hline 
$\frac{\textrm{UB}-\textrm{L\ensuremath{B_{0.95}}}}{|\textrm{UB}|}$(\%) & 15.44 & 14.16 & 14.59 & 11.96 & 10.70 & 9.84\tabularnewline
\hline 
$\frac{\textrm{UB-Opt.Obj}}{\textrm{|Opt.Obj|}}$(\%) & 8.44 & 9.58 & 9.08 & 9.57 & 8.53 & 9.40\tabularnewline
\hline 
UB Time (s) & $\textrm{324.842}$ & $\textrm{392.786}$ & 378.193 & 379.355 & 379.424 & 438.124\tabularnewline
\hline 
LB Time (s) & $\textrm{84.837}$ & $\textrm{80.652}$ & 86.2818 & 86.3863 & 98.4709 & 469.205\tabularnewline
\hline 
Total Time (s) & 409.679 & 473.438 & 465.075 & 465.741 & 477.895 & 907.329\tabularnewline
\hline 
\end{tabular}
\par\end{centering}
\caption{Validation of optimal solution with sample size}

\end{table}

\subsection{Online learning}

This experiment seek to check the convergence performance for Algorithms
2. Set random outcome $X\sim N(100,\,50)$ and target $\tau=120$;
For general risk measure, we choose another conjugate of divergence
function as Kullback-Leibler divergence i.e., $\phi^{*}(s)=\frac{1}{\alpha}\left[\exp(s)-1\right]$
and $\rho=50$. Both the online algorithms and show great convergence
in terms of regret. The results are shown in following Figure 2. Both
the regret of Algorithm 2 converges within 1000 steps, which demonstrate
that online optimization will produce a estimated risk index close
to the true one.

\begin{figure}

\begin{centering}
\includegraphics[scale=0.5]{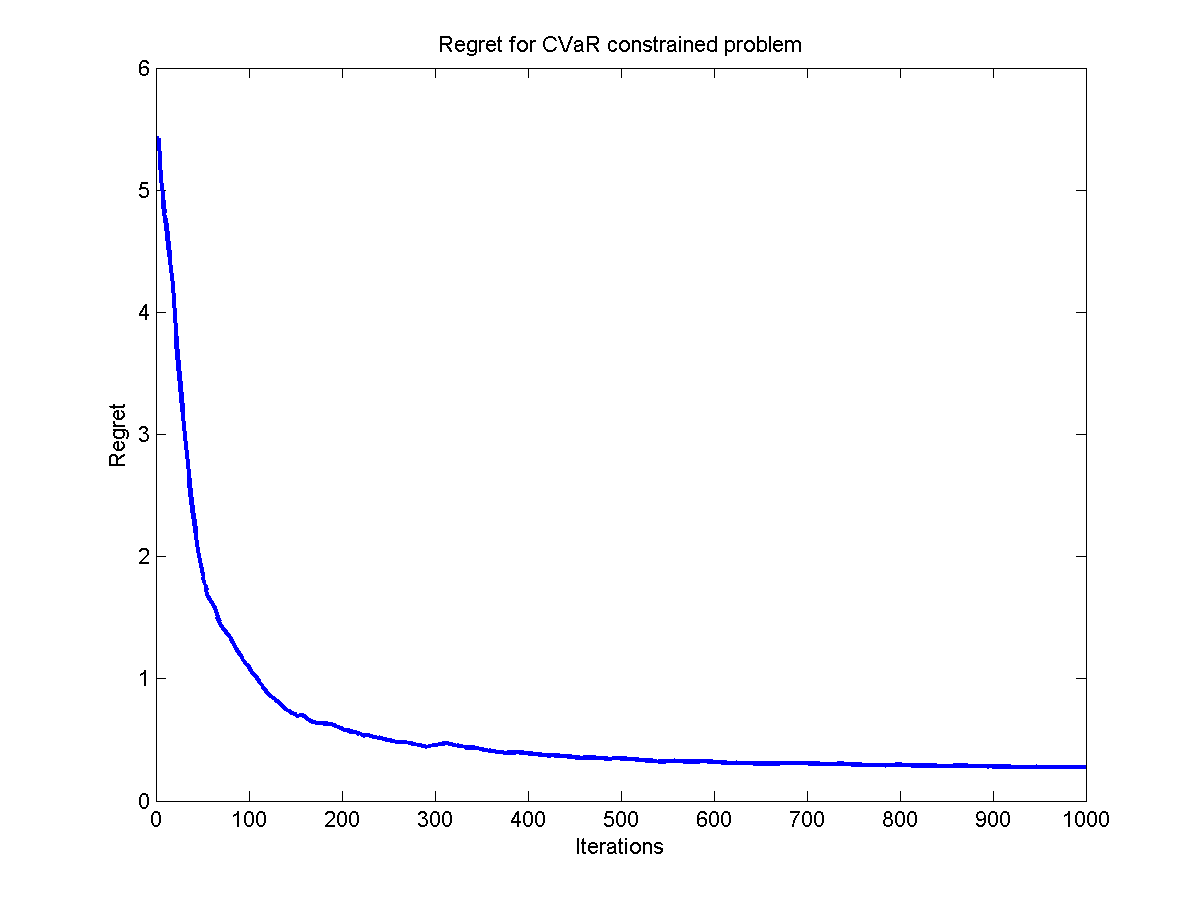}
\par\end{centering}
\begin{centering}
\includegraphics[scale=0.5]{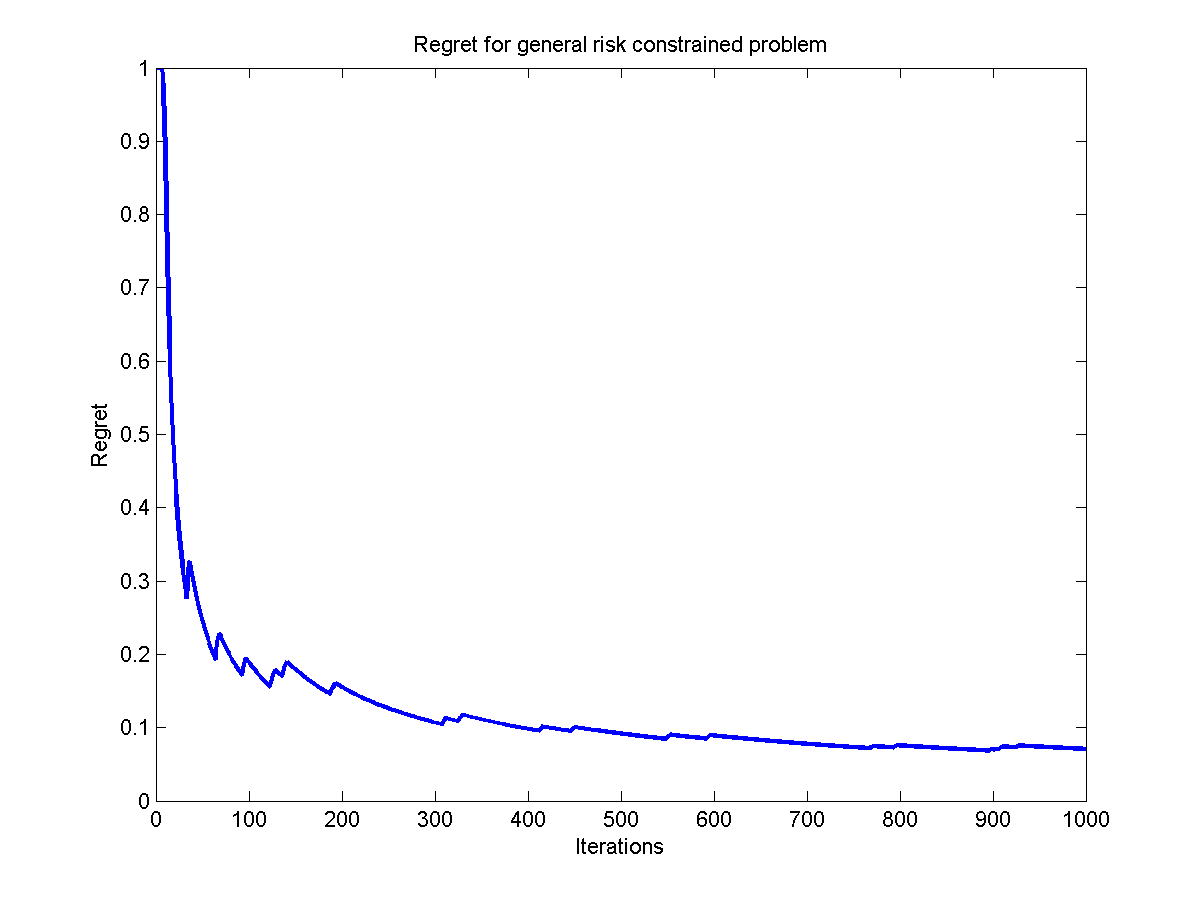}
\par\end{centering}
\caption{Regret for CVaR and general risk constrained problem}

\end{figure}
Here we conduct numerical experiments seeking compare the practical
probability bounds for offline risk ranking by simulation and theoretical
bounds provided by Theorem \ref{Theorem 5.4-1}. We conduct the online
risk ranking with Conditional value-at-risk as risk measure for 8
random variables follow $X\sim N(100,\,50)$ with target respectively
as $\tau=115,\,117,\,119,\,121,\,123,\,125,\,127,\,129$. The following
figure shows the relationship between the validation probability of
ranking results with the iteration number. The maximum possible inversion
number for this ranking system is $C_{8}^{2}=28$. The experiments
are conducted under the case when inversion number $e$ equals $5,\,10,\,15,\,20$
and $25$ respectively. The $\delta$ equal $0.05$ for all the cases.
Figure 3 shows that the practical bound attained by simulation is
lower bound by the theoretical bounds derive in Theorem \ref{Theorem 5.4-1},
thus the practical bound could serve as a conservative guidance on
determined the required size of the sample to ensure a high probability
of risk ranking validation.

\begin{figure}
\begin{centering}
\includegraphics[scale=0.48]{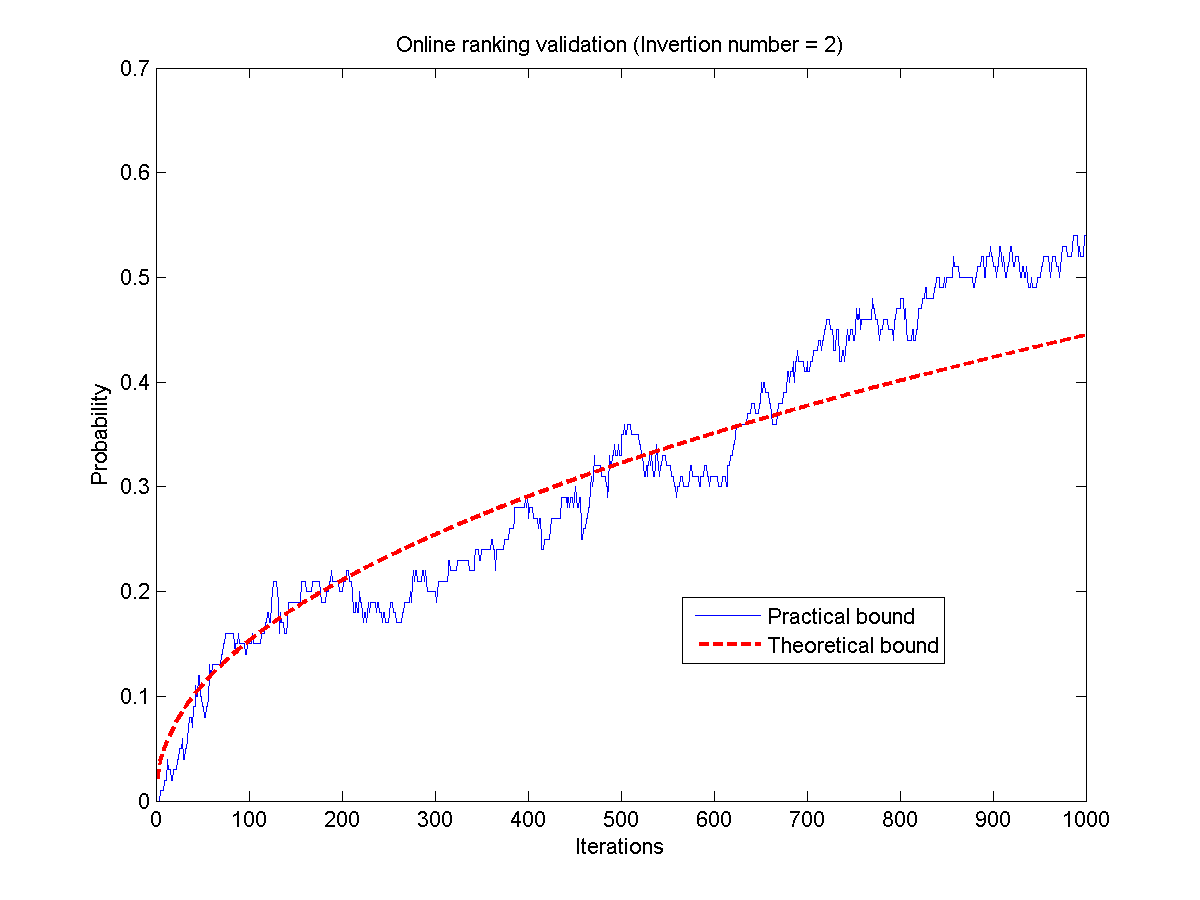}
\par\end{centering}
\begin{centering}
\includegraphics[scale=0.48]{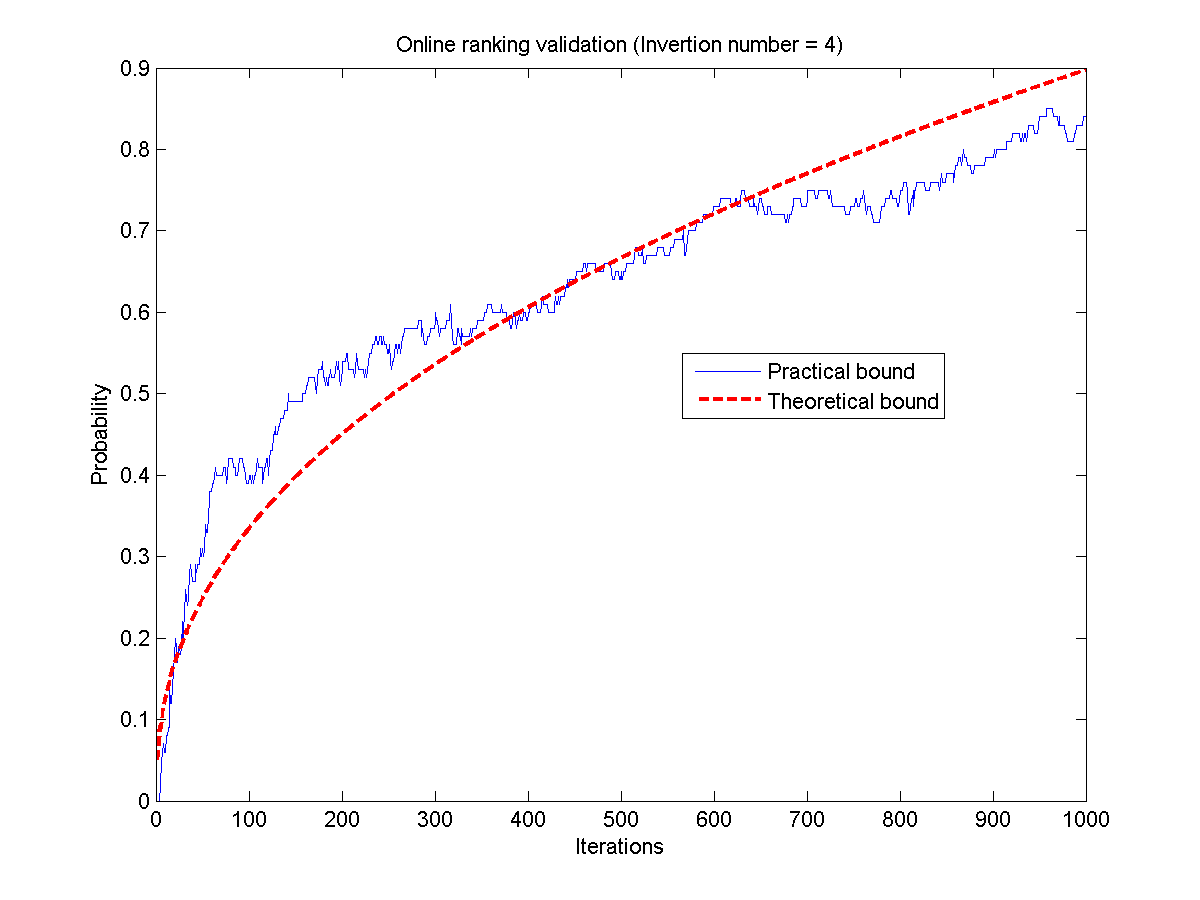}
\par\end{centering}
\begin{centering}
\includegraphics[scale=0.48]{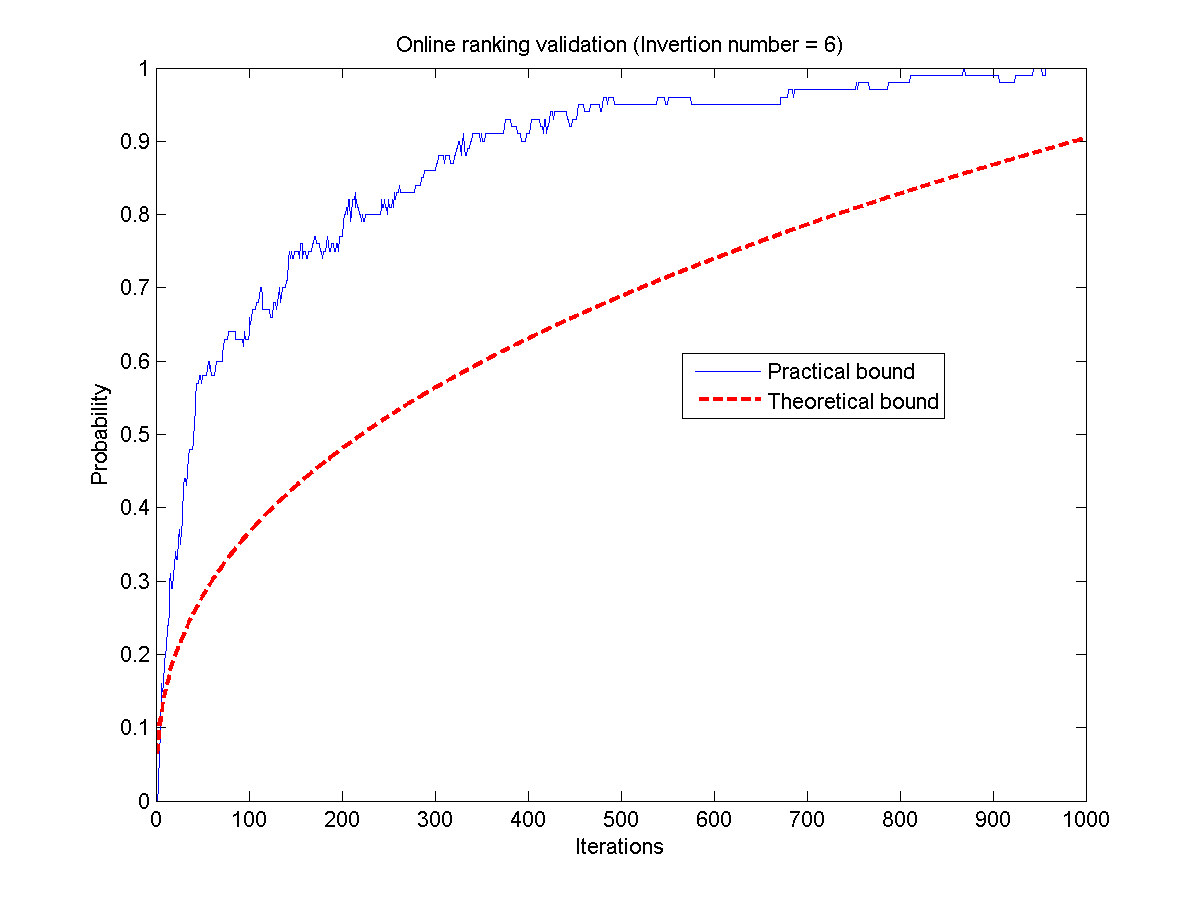}
\par\end{centering}
\caption{Online ranking validation}
\end{figure}

\subsection{Risk assessment and ranking results}

The following figure reveals the real-time risk assessment and ranking
results with 1000 iterations of streaming concentration data of 13
EOCs. Here we use the Conditional value-at-risk as risk measure. From
Figure 3, we can conclude that the Estrone, BPA, Triclosan, and Fipronil
are among the most risky EOCs, and Salicylic acid, Sulpiride and Naproxen
are among the least risky EOC. The risk level of other EOCs are intermediate.
Starting from around 800 iterations, our ranking results becomes stable.

\begin{figure}
\begin{centering}
\includegraphics[scale=0.8]{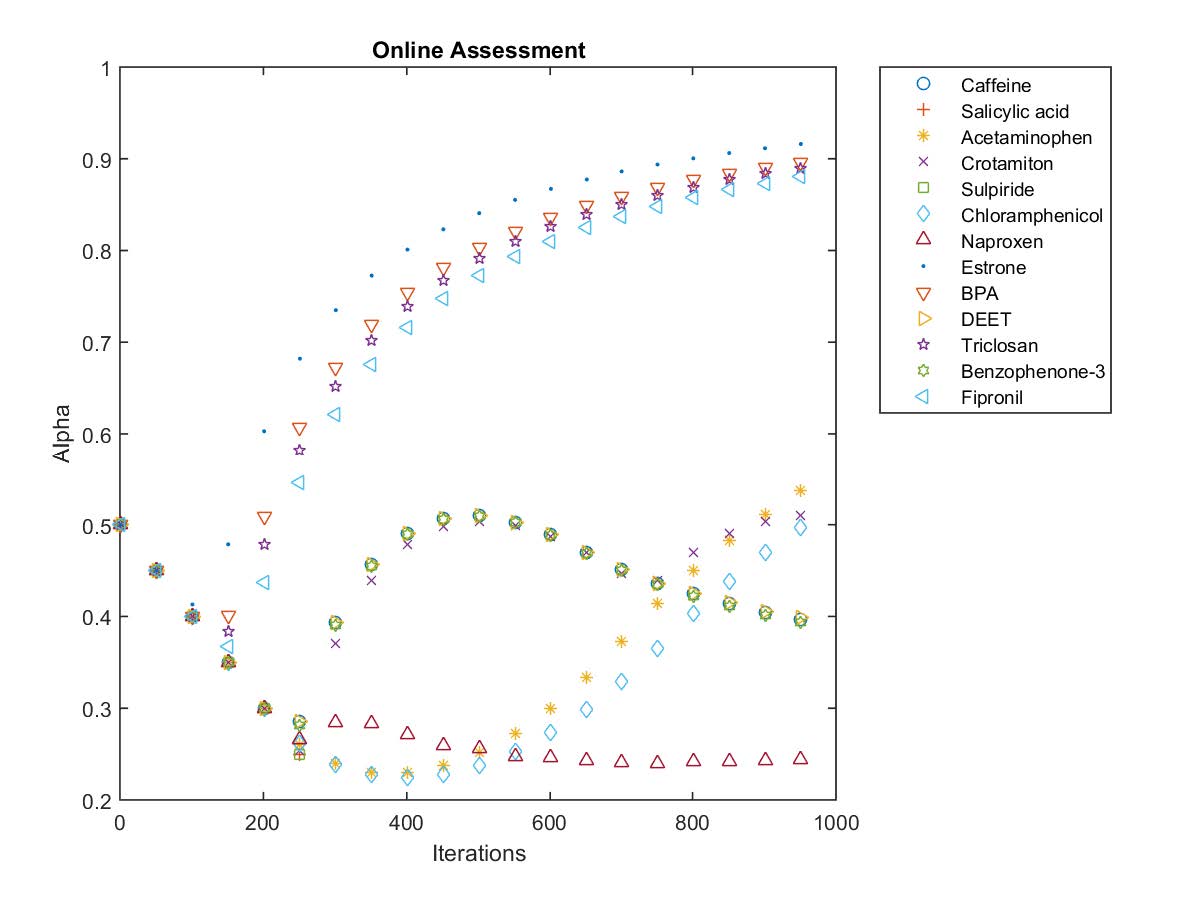}
\par\end{centering}
\caption{Online assessment and ranking result for EOCs}

\end{figure}
The following table shows the comparison online learning method and
batch learning method. We observe that our online and offline computational
methods produce similar ranking and assessment results, when online
ranking is more sensitive compared to offline ranking, which will
contribute more in capture the silent difference between the risk
level of EOCs. 
\begin{center}
\begin{table}
\begin{centering}
\begin{tabular}{|c|c|c|c|c|}
\hline 
Contaminant  & $\alpha$-offline  & ranking-offline & $\alpha$-online & ranking-online\tabularnewline
\hline 
\hline 
Caffeine & 6.1035e-05 & 7 & 0.3897 & 9\tabularnewline
\hline 
Salicylic acid  & 0.0000 & 11 & 0.0000 & 12\tabularnewline
\hline 
Acetaminophen  & 0.5308 & 5 & 0.5598 & 5\tabularnewline
\hline 
Crotamiton  & 6.1035e-05 & 7 & 0.5106 & 7\tabularnewline
\hline 
Sulpiride  & 0.0000 & 11 & 0.0000 & 12\tabularnewline
\hline 
Chloramphenicol  & 0.3050 & 6 & 0.5221 & 6\tabularnewline
\hline 
Naproxen  & 6.1035e-05 & 7 & 0.2455 & 11\tabularnewline
\hline 
Estrone  & 0.7189 & 1 & 0.9202 & 1\tabularnewline
\hline 
BPA  & 0.7189 & 1 & 0.9014 & 2\tabularnewline
\hline 
DEET  & 6.1035e-05 & 7 & 0.3910 & 8\tabularnewline
\hline 
Triclosan  & 0.7189 & 1 & 0.8952 & 3\tabularnewline
\hline 
Benzophenone-3  & 6.1035e-05 & 7 & 0.3869 & 10\tabularnewline
\hline 
Fipronil  & 0.7189 & 1 & 0.8860 & 4\tabularnewline
\hline 
\end{tabular}
\par\end{centering}
\caption{Comparison of offline and online assessment and ranking }

\end{table}
\par\end{center}

\section{Conclusion}

This paper proposes an risk assessment and ranking approach based
on satisficing measure, and develops real-time risk assessment and
ranking policy with online learning and optimization technique. The
basic model is a satisficing optimization model with a family of risk
constrained (Conditional value-at-risk or Optimized certainty equivalent).
In the offline case, we consider performing risk assessment and ranking
when we know the probability distribution of random outcomes or we
have no known information about the probability distribution but have
samples extracted from the underlying distribution in hand. Sample
average approximation can attack both cases. We figure out the convergence
rate with sample size and prove the validation of ranking result from
approximated problem under offline case. For online stochastic optimization
case, we develop an online primal-dual algorithm to solve the problem
with a desired regret bound. For both offline and online case, we
argue the minimal sample size or iteration times required to ensure
a high probability that the estimated ranking is the same as the true
ranking, as well as, given a certain sample size or iteration times,
the probability that the loss function is bounded by a given benchmark. 

\section*{Acknowledgment}

This research is funded by the National Research Foundation (NRF),
Prime Ministers Office, Singapore under its Campus for Research Excellence
and Technological Enterprise (CREATE) program.\bibliographystyle{apalike}
\bibliography{E2S2-Reference,References}

\section*{Appendix I }
\begin{center}
\begin{tabular}{|c|c|c|}
\hline 
\textbf{Name} & $\phi(t)\quad t\geq0$ & $\phi^{*}(s)$\tabularnewline
\hline 
\hline 
Kullback-Leibler & $t\log t-t+1$ & $\exp(s)-1$\tabularnewline
\hline 
Burg entropy & $-\log t+t-1$ & $-\log(1-s),\quad s<1$\tabularnewline
\hline 
$\chi^{2}$distance & $\frac{1}{t}(t-1)^{2}$ & $2-2\sqrt{1-s},\quad s<1$\tabularnewline
\hline 
Modified$\chi^{2}$distance & $(t-1)^{2}$ & $\begin{cases}
-1 & \quad s<-2\\
s+s^{2}/4 & \quad s\geq-2
\end{cases}$\tabularnewline
\hline 
Hellinger distance & $(\sqrt{t}-1)^{2}$ & $\frac{s}{1-s},\quad s<1$\tabularnewline
\hline 
$\chi$- divergence & $|t-1|^{\theta}$ & $s+(\theta+1)\left(\frac{|s|}{\theta}\right)^{\theta/(\theta-1)}$\tabularnewline
\hline 
Variation distance & $|t-1|$ & $\max\left\{ -1,\,s\right\} ,\quad s\leq1$\tabularnewline
\hline 
Cressie-Read & $\frac{1-\theta+\theta t-t^{\theta}}{\theta(1-\theta)},\quad t\neq0,1$ & $\frac{1}{\theta}(1-s(1-\theta))^{\theta/(1-\theta)}-\frac{1}{\theta},\quad s<\frac{1}{1-\theta}$\tabularnewline
\hline 
\end{tabular}
\par\end{center}

\begin{center}
Table 4: Examples of $\phi-$divergence functions and their convex
conjugate functions
\par\end{center}

\section*{Appendix II}

\noun{Proof of Theorem \ref{Theorem 4.6}}: Problem (\ref{General problem-1})
can be reformulated as
\begin{align*}
\max_{0\leq\alpha\leq1,\,\eta\in\mathbb{R}} & \left\{ 1-\alpha:\,g(\alpha):=\mathbb{E}^{P}\left[G\left(\alpha,\,\eta,\,X\right)\right]\leq\tau\right\} ,
\end{align*}
and Problem (\ref{SAA}) can be reformulated as
\begin{align*}
\max_{0\leq\alpha\leq1,\,\eta\in\mathbb{R}} & \left\{ 1-\alpha:\,g_{N}(\alpha):=\frac{1}{N}\sum_{n=1}^{N}G\left(\alpha,\,\eta,\,d_{n}\right)\leq\tau\right\} ,
\end{align*}
Given $\epsilon>0$ define
\[
X^{\epsilon}:=\left\{ \alpha\in\left[0,\,1\right]:\,g(\alpha)\leq\tau+\epsilon\right\} .
\]
Then $X^{0}$ represents the feasible region of Problem (\ref{General problem-1}),
and correspondingly we define
\[
X_{N}^{\epsilon}:=\left\{ \alpha\in\left[0,\,1\right]:\,g_{N}(\alpha)\leq\tau+\epsilon\right\} .
\]
Our goal is to estimate $\mathbb{P}\left\{ X^{-\epsilon}\subseteq X_{N}^{0}\subseteq X^{\epsilon}\right\} .$
That is, we want to claim a feasible solution of SAA of Problem (\ref{SAA})
is $\epsilon-$feasible to the true problem. Suppose Assumption \ref{Assumption 4.1}
hold. Given $\epsilon>0,$ we could set $v:=\left\{ 4(\Psi+\Phi)/\epsilon+2\right\} ^{-1}$,
and obtain
\begin{align*}
 & \mathbb{P}\left\{ X^{-\epsilon}\subseteq X_{N}^{0}\subseteq X^{\epsilon}\right\} \\
\geq & \mathbb{P}\left\{ \exists\alpha\in(0,\,1],\,\textrm{s.t.}\,\Bigl|g_{N}(\alpha)-g(\alpha)\Bigr|\leq\epsilon\right\} \\
\geq & 1-\mathbb{P}\left\{ \exists\alpha\in(0,\,1],\,\textrm{s.t.}\,\sup_{\eta\in\mathbb{R}}\left\{ \biggl|\mathbb{E}^{P}\left[\frac{1}{\alpha}\phi^{\ast}\left(X-\eta\right)\right]-\frac{1}{N}\sum_{n=1}^{N}\frac{1}{\alpha}\phi^{\ast}\left(d_{n}-\eta\right)\biggr|\right\} \geq\epsilon\right\} \\
\geq & 1-\mathbb{P}\left\{ \exists\alpha\in(0,\,1]_{v},\,\textrm{s.t.}\,\sup_{\eta\in E_{v}}\mathbb{E}^{P}\left[\frac{1}{\alpha}\phi^{\ast}\left(X-\eta\right)\right]-\frac{1}{N}\sum_{n=1}^{N}\frac{1}{\alpha}\phi^{\ast}\left(d_{n}-\eta\right)\geq\epsilon-(\Psi+\Psi_{N}+\Phi+\Phi_{N})v\right\} \\
 & -\mathbb{P}\left\{ \exists\alpha\in(0,\,1]_{v},\,\textrm{s.t.}\,\sup_{\eta\in E_{v}}\frac{1}{N}\sum_{n=1}^{N}\frac{1}{\alpha}\phi^{\ast}\left(d_{n}-\eta\right)-\mathbb{E}^{P}\left[\frac{1}{\alpha}\phi^{\ast}\left(X-\eta\right)\right]\geq\epsilon-(\Psi+\Psi_{N}+\Phi+\Phi_{N})v\right\} ,
\end{align*}
Then we have
\begin{align*}
 & \mathbb{P}\left\{ X^{-\epsilon}\subseteq X_{N}^{0}\subseteq X^{\epsilon}\right\} \\
\geq & 1-2\mathbb{P}\left\{ \Psi_{N}>\Psi+\epsilon/2\right\} -2\mathbb{P}\left\{ \Phi_{N}>\Phi+\epsilon/2\right\} \\
 & -2\sum_{\alpha\in(0,\,1],\,\eta\in E_{v}}\mathbb{P}\left\{ \biggl|\mathbb{E}^{P}\left[\frac{1}{\alpha}\phi^{\ast}\left(X-\eta\right)\right]-\frac{1}{N}\sum_{n=1}^{N}\frac{1}{\alpha}\phi^{\ast}\left(d_{n}-\eta\right)\biggr|\geq\epsilon/2\right\} \\
\geq & 1-2\exp\left\{ -N\cdot b(\epsilon)\right\} -\sum_{\alpha\in(0,\,1],\,\eta\in E_{v}}\left[\exp\left\{ -N\cdot b(\epsilon)\right\} \right]\\
\geq & 1-2\left(2+\frac{D_{E}}{v^{2}}\right)\exp\left\{ -N\cdot b(\epsilon)\right\} ,
\end{align*}
where $b(\epsilon)\,:=\min_{\eta\in E_{v}}\left\{ I_{\psi}(\Psi+\epsilon/2),\,I_{\phi}(\Phi+\epsilon/2),\,I_{\eta}(\epsilon/2),\,I_{\eta}(-\epsilon/2)\right\} $,
and we have
\begin{align*}
b(\epsilon) & \geq\min_{\alpha\in(0,1],\,\eta\in E_{v}}\left\{ \frac{\epsilon^{2}}{8\textrm{VaR}\left[\phi(X)\right]},\,\frac{\epsilon^{2}}{8\textrm{VaR}\left[\psi(X)\right]},\,\frac{\epsilon^{2}}{\textrm{VaR}\left[G(\alpha,\,\eta,\,X)-\mathbb{E}\left[G(\alpha,\,\eta,\,X)\right]\right]}\right\} \\
 & \geq\frac{\epsilon^{2}}{8\sigma^{2}}.
\end{align*}
Assume that there exists a positive number $C$ so that, $\tilde{\pi}\leq C$
uniformly, and $\epsilon_{0}$ that $\|\tau-\tilde{\tau}\|_{2}\leq\epsilon_{0}$.
Given $0<\epsilon\leq c,$ we set $v:=\left\{ 4(\Psi+\Phi)C/(\epsilon-C\epsilon_{0})+2\right\} ^{-1}$.
It follows that
\begin{align*}
 & \mathbb{P}\left\{ |\tilde{\alpha}-\alpha|\leq\epsilon\right\} \\
= & \mathbb{P}\left\{ |1-\tilde{\alpha}-(1-\alpha)|\leq\epsilon\right\} \\
\geq & \mathbb{P}\left\{ \biggl|\sup_{0\leq\alpha\leq1,\,t\in C}\left\{ 1-\alpha+\tilde{\pi}\left[g_{N}(\alpha)-\tilde{\tau}\right]\right\} -\sup_{0\leq\alpha\leq1,\,t\in C}\left\{ 1-\alpha+\tilde{\pi}\left[g(\alpha)-\tau\right]\right\} \biggr|\leq\epsilon\right\} \\
= & \mathbb{P}\left\{ \sup_{0\leq\alpha\leq1}\biggl|\tilde{\pi}\left[g_{N}(\alpha)-g(\alpha)\right]\biggr|+|\tilde{\pi}\left(\tau-\tilde{\tau}\right)|\leq\epsilon\right\} \\
= & \mathbb{P}\left\{ \tilde{\pi}\left[\sup_{0\leq\alpha\leq1}\biggl|g_{N}(\alpha)-g(\alpha)\biggr|+|\tau-\tilde{\tau}|\right]\leq\epsilon\right\} \\
\geq & \mathbb{P}\left\{ \sup_{0\leq\alpha\leq1}\biggl|g_{N}(\alpha)-g(\alpha)\biggr|\leq\frac{\epsilon}{C}-\epsilon_{0}\right\} \\
= & 1-\mathbb{P}\left\{ \sup_{0\leq\alpha\leq1}\biggl|g_{N}(\alpha)-g(\alpha)\biggr|\geq\frac{\epsilon}{C}-\epsilon_{0}\right\} \\
\geq & 1-\mathbb{P}\left\{ \exists0\leq\alpha\leq1\,\textrm{s.t.}\,g_{N}(\alpha)-g(\alpha)>\frac{\epsilon}{C}-\epsilon_{0}\right\} \\
 & -\mathbb{P}\left\{ \exists0\leq\alpha\leq1,\,\textrm{s.t.}\,g_{N}(\alpha)-g(\alpha)<-\frac{\epsilon}{C}+\epsilon_{0}\right\} \\
\geq & \mathbb{P}\left\{ |\tilde{\alpha}-\alpha|\leq\epsilon\right\} \geq1-2\left(2+\frac{D_{E}}{v^{2}}\right)\exp\left(-\frac{N(\epsilon-C\epsilon_{0})^{2}}{8\sigma^{2}C^{2}}\right),
\end{align*}
where
\[
v:=\left\{ 4(\Psi+\Phi)C/(\epsilon-C\epsilon_{0})+2\right\} ^{-1},
\]
and
\[
\sigma^{2}:=\text{\ensuremath{\max}}_{\alpha\in(0,1],\,\eta\in E}\left\{ \textrm{VaR}\left[\phi(X)\right],\,\textrm{VaR}\left[\psi(X)\right],\,\textrm{VaR}\left[G(\alpha,\,\eta,\,X)-\mathbb{E}\left[G(\alpha,\,\eta,\,X)\right]\right]\right\} .
\]
In addition, we have
\begin{align*}
 & \mathbb{P}\left\{ |\tilde{l}_{L}-(1-\alpha)|\leq c-\epsilon\right\} \\
= & \mathbb{P}\left\{ |\widetilde{l}-z_{\gamma/2}S_{\tilde{l}}-(1-\alpha)|\leq c-\epsilon\right\} \\
\geq & \mathbb{P}\left\{ \Biggl|\frac{1}{M_{l}}\sum_{m=1}^{M_{l}}\inf_{\alpha\in(0,\,1]}\tilde{\pi}\left[\sup_{u\geq0,\,\eta\in\mathbb{R}}N_{u}^{-1}\sum_{n=1}^{N_{u}}\left\{ \eta+u\mu+\frac{u}{\alpha}\phi^{*}\left(\frac{d_{n}^{m}-\eta}{u}\right)\right\} -\tau\right]\Biggr|+|z_{\gamma/2}S_{\tilde{l}}|\leq c-\epsilon\right\} \\
\geq & \mathbb{P}\left\{ \tilde{\pi}\sup_{\alpha\in\left[0,\,1\right]}\biggl|g_{N}(\alpha)-g(\alpha)\biggr|+|z_{\gamma/2}S_{\tilde{l}}|\leq c-\epsilon\right\} \\
\geq & \mathbb{P}\left\{ \tilde{\pi}\sup_{\alpha\in\left[0,\,1\right]}\biggl|g_{N}(\alpha)-g(\alpha)\biggr|\leq c-\epsilon-z_{\gamma/2}\cdot\frac{M^{2}}{4}\right\} \\
\geq & 1-2\left(2+\frac{D_{E}}{v^{2}}\right)\exp\left\{ \frac{-N\cdot(c-\epsilon-z_{\gamma/2}\cdot\frac{M^{2}}{4})^{2}}{8\sigma^{2}}\right\} ,
\end{align*}
where
\[
v:=\left\{ 4(\Psi+\Phi)/(c-\epsilon-z_{\gamma/2}\cdot\frac{M^{2}}{4})+2\right\} ^{-1},
\]
and 
\[
\sigma^{2}:=\text{\ensuremath{\max}}_{\alpha\in(0,1],\,\eta\in E}\left\{ \textrm{VaR}\left[\phi(X)\right],\,\textrm{VaR}\left[\psi(X)\right],\,\textrm{VaR}\left[G(\alpha,\,\eta,\,X)-\mathbb{E}\left[G(\alpha,\,\eta,\,X)\right]\right]\right\} .
\]
Based on above arguments, next we study the relationship between the
sample size and the validation probability of ranking by SAA problem
to the original problem.

\section*{Appendix III}

\noun{Proof of Theorem \ref{Theorem 5.3}}: In this section, we will
prove Theorem \ref{Theorem 5.3}. Denote $A_{t}:=\frac{1}{2}\|\alpha_{t}-\alpha^{\ast}\|_{2}^{2}$
. We can write
\begin{align*}
A_{t}= & \frac{1}{2}\|\text{\ensuremath{\prod}}_{(0,\,1]}\left(\alpha_{t-1}-\frac{G(\alpha_{t-1},\,\eta_{t-1},\,\lambda_{t-1})}{t}\right)-\alpha^{\ast}\|_{2}^{2}\\
\leq & \frac{1}{2}\|\alpha_{t-1}-\frac{G(\alpha_{t-1},\,\eta_{t-1},\,\lambda_{t-1})}{t}-\alpha^{\ast}\|_{2}^{2}\\
= & A_{t}+\frac{1}{2t^{2}}\|G(\alpha_{t-1},\,\eta_{t-1},\,\lambda_{t-1})\|_{2}^{2}\\
 & +\frac{1}{t}(\alpha_{t-1}-\alpha^{\ast})^{\top}G(\alpha_{t-1},\,\eta_{t-1},\,\lambda_{t-1},)
\end{align*}
and we also have
\begin{align*}
\mathbb{E}\left[(\alpha_{t-1}-\alpha^{\ast})^{\top}G(\alpha_{t-1},\,\eta_{t-1},\,\lambda_{t-1})\right]\\
=\mathbb{E}_{d_{t-2}}\left\{ \mathbb{E}_{d_{t-1}}\left[(\alpha_{t-1}-\alpha^{\ast})^{\top}G(\alpha_{t-1},\,\eta_{t-1},\,\lambda_{t-1})\right]\right\} \\
=\mathbb{E}_{d_{t-2}}\left\{ (\alpha_{t-1}-\alpha^{\ast})^{\top}\mathbb{E}\left[G(\alpha_{t-1},\,\eta_{t-1},\,\lambda_{t-1})\right]\right\} \\
=\mathbb{E}\left[(\alpha_{t-1}-\alpha^{\ast})^{\top}g(\alpha_{t-1},\,\eta_{t-1},\,\lambda_{t-1})\right].
\end{align*}
By taking the expectation of both side, we obtain
\[
\frac{1}{2}\mathbb{E}\left[\|\alpha_{t}-\alpha^{\ast}\|_{2}^{2}\right]\leq\frac{1}{2}\mathbb{E}\left[\|\alpha_{t-1}-\alpha^{\ast}\|_{2}^{2}\right]+\frac{1}{t}\mathbb{E}\left[(\alpha_{t-1}-\alpha^{\ast})^{\top}g(\alpha_{t},\,\eta_{t},\,\lambda_{t})\right]+\frac{1}{2t^{2}}M^{2},
\]
where, based on Assumption \ref{Assumption 5.2} (i), obtain
\[
M^{2}:=\sup_{(\alpha,\,\eta)\in(0,1],\,\lambda\leq0}\mathbb{E}\left[\|G(\alpha_{t-1},\,\eta_{t-1},\,\lambda_{t-1})\|_{2}^{2}\right].
\]
By Lipschitz continuous of function $g$ summarized in Assumption
\ref{Assumption 5.2} (ii), we have
\[
\mathbb{E}\left[(\alpha_{t}-\alpha^{\ast})^{\top}g(\alpha_{t},\,\eta_{t},\,\lambda_{t})\right]\geq\mathbb{E}\left[(\alpha_{t}-\alpha^{\ast})^{\top}\left(g(\alpha_{t},\,\eta_{t},\,\lambda_{t})-g(\alpha^{\ast},\,\eta_{t},\,\lambda_{t})\right)\right]\geq L_{g}\mathbb{E}\left[\|\alpha_{t}-\alpha^{\ast}\|_{2}^{2}\right].
\]
Therefore it follows that
\[
\mathbb{E}\left[\|\alpha_{t}-\alpha^{\ast}\|_{2}^{2}\right]\leq(1-2L_{g}/t)\mathbb{E}\left[\|\alpha_{t-1}-\alpha^{\ast}\|_{2}^{2}\right]+M^{2}/t^{2},
\]
and by induction we get the deterministic convergence rate
\[
\mathbb{E}\left[\|\alpha_{t}-\alpha^{\ast}\|_{2}^{2}\right]\leq\kappa/t,
\]
where
\[
\kappa:=\max\left\{ M^{2}/(2L_{g}-1)^{-1},\,\|\alpha^{0}-\alpha^{\ast}\|_{2}^{2}\right\} .
\]
Follow the same idea ,we obtain
\[
\mathbb{E}\left[\|\alpha_{t}-\alpha^{\ast}\|_{2}^{2}+\|\eta_{t}-\eta^{\ast}\|_{2}^{2}+\|\lambda_{t}-\lambda^{\ast}\|_{2}^{2}\right]\leq3\kappa/t,
\]
where
\[
\kappa:=\max\left\{ M^{2}/(2L_{g}-1)^{-1},\,\|\alpha_{0}-\alpha^{\ast}\|_{2}^{2},\,\|\eta_{0}-\eta^{\ast}\|_{2}^{2},\,\|\lambda_{0}-\lambda^{\ast}\|_{2}^{2}\right\} ,
\]
Let $\delta_{t}=\left[L(\alpha^{\ast},\,\eta^{\ast},\,\lambda^{\ast})\right](d_{t})-C^{\ast}$
and $\xi_{t}=\left[L(\alpha_{t-1},\,\eta_{t-1},\,\lambda_{t-1})\right](d_{t})-\left[L(\alpha^{\ast},\,\eta^{\ast},\,\lambda^{\ast})\right](d_{t})$
be the statistical error and approximation error, where $\alpha^{\ast},\,\eta^{\ast},\,\lambda^{\ast}$
denote the optimal primal and dual solution for problem \ref{Saddle}.
Therefore the step 4 in Algorithm 2 is equivalent to
\[
C_{t}=C_{t-1}-\frac{1}{t}\left(C_{t-1}-C^{\ast}+\xi_{t}+\delta_{t}\right).
\]
Expanding, we have following inequalities
\begin{align*}
\|C_{t}-C^{\ast}\|_{2}^{2} & \leq\|C_{t-1}-C^{\ast}\|_{2}^{2}+\|\frac{1}{t}\left(C_{t-1}-C^{\ast}+\xi_{t}+\delta_{t}\right)\|_{2}^{2}\\
 & -2(C_{t-1}-C^{\ast})^{\top}\frac{1}{t}\left(C_{t-1}-C^{\ast}+\xi_{t}+\delta_{t}\right),
\end{align*}
\[
\mathbb{E}\left[C_{t-1}-\left[L(\alpha_{t-1},\,\eta_{t-1},\,\lambda_{t-1})\right](d_{t})\right]^{2}\leq B,
\]
\[
\|C_{t}-C^{\ast}\|_{2}^{2}\leq\|C_{t-1}-C^{\ast}\|_{2}^{2}+\frac{B-2}{t^{2}}-\frac{2}{t}\kappa_{0}\|C_{t-1}-C^{\ast}\|_{2}^{2}-\frac{L_{\Phi}}{t\kappa_{0}}\left(\|\alpha_{t}-\alpha^{\ast}\|_{2}^{2}+\|\eta_{t}-\eta^{\ast}\|_{2}^{2}+\|\lambda_{t}-\lambda^{\ast}\|_{2}^{2}\right),
\]
and
\[
\|C_{t}-C^{\ast}\|_{2}^{2}\leq(1-2\kappa_{0}/t)\|C_{t-1}-C^{\ast}\|_{2}^{2}+\frac{B-2-\frac{3\kappa L_{\Phi}}{\kappa_{0}}}{t^{2}}.
\]
Then we have
\[
\mathbb{E}\left[\|C_{t}-C^{\ast}\|_{2}^{2}\right]\leq\kappa^{\prime}/t,
\]
where
\[
\kappa^{\prime}:=\max\left\{ B-2-\frac{3\kappa L_{\Phi}}{\kappa_{0}}/(2\kappa_{0}-1)^{-1},\,\|C_{0}-C^{\ast}\|_{2}^{2}\right\} ,
\]
Thus we get the desired result.
\end{document}